\documentclass[a4paper,11pt]{article}

\usepackage{amsmath}    % need for subequations
\usepackage{amssymb}
\usepackage{amsthm} 
\usepackage{bbm}
\usepackage{booktabs}
\usepackage{microtype}
\usepackage{a4wide}
\usepackage{todonotes}

\newtheorem{theorem}{Theorem}[section]
\newtheorem{corollary}[theorem]{Corollary}
\newtheorem{lemma}[theorem]{Lemma}
\newtheorem{proposition}[theorem]{Proposition}

\theoremstyle{definition}
\newtheorem{definition}[theorem]{Definition}
\newtheorem{remark}[theorem]{Remark}

\usepackage[toc,page]{appendix}
\usepackage{color}
\usepackage{dsfont}
\usepackage{extarrows}
\usepackage{enumerate}
\usepackage{graphicx}   % need for figures
\usepackage{mathrsfs}
\usepackage{verbatim}   % useful for program listings
\usepackage{subfigure}  % use for side-by-side figures
\usepackage{hyperref}   % use for hypertext links, including those to external documents and URLs
\DeclareMathOperator{\Lip}{Lip}

\usepackage{mathabx}

\newcommand{\Wc}{W} % Wasserstein distance (order determined by subscript)

\def\beq{\begin{eqnarray}}
\def\eeq{\end{eqnarray}}
\def\be*{\begin{eqnarray*}}
\def\ee*{\end{eqnarray*}}

\def \I{\mathbb{I}}

\def \N{\mathbb{N}}

\def \P{\mathbb{P}}
\def \Q{\mathbb{Q}}
\def \R{\mathbb{R}}

\def\Ac{{\cal A}}

\def\Fc{{\cal F}}

\def\Mc{{\cal M}}

\def\Pc{{\cal P}}

\def\Sc{{\cal S}}

\def\Wc{{\cal W}}

\def \d {\mbox{d}}
\def \eps {\varepsilon}

\def \om {\omega}

\def \cv {\le_{\rm c}}
\def \cvd {\le_{\rm cd}}

\def \AW {\Ac\Wc_1}

\def\namedlabel#1#2{\begingroup
	#2%
	\def\@currentlabel{#2}%
	\phantomsection\label{#1}\endgroup
}

\newcommand{\sP}{\Pc} % alias: space of probability measures

\def \supp{{\rm{supp}}}
\def \1 {{\mathbf 1}}

\def \ba{{\textsf{bary}}}

\title{\Large \bf
Stability of weak supermartingale optimal transport problems
}
\date{\vspace{-5ex}}

\author{Shuoqing Deng
\thanks{
{\small The Hong Kong University of Science and Technology, Department of Mathematics (masdeng@ust.hk). S. Deng is supported by the Hong Kong University of Science and Technology Start-up Grant No. R9826 and Hong Kong RGC Early Career Scheme (ECS) Grant No. 26307125. }}%
\and 
Gaoyue Guo
\thanks{
{\small 
Université Paris-Saclay CentraleSupélec, Laboratoire MICS and CNRS FR-3487 (gaoyue.guo@centralesupelec.fr). Guo acknowledges financial support of ANR-25-CE40-0714 (MATH-SPA).
}}%}
\and
Dominykas Norgilas
\thanks{
{\small 
North Carolina State University, Department of Mathematics (dnorgil@ncsu.edu).}}%}
}

\begin{document}

% \begin{center}
% {\Large Stabilization of Jump-Diffusion stochastic differential equation by hysteresis switching} 
% \end{center}
% \begin{center}
% {\large \today}\\[3cm] 
% \end{center}

\maketitle

\begin{abstract}
We investigate stability properties of weak supermartingale optimal transport (WSOT) problems on $\R$. 
For probability measures $\mu,\nu\in\Pc_r$ satisfying $\mu\cvd \nu$ (equivalently, $\Pi_S(\mu,\nu)\neq\emptyset$), we consider supermartingale couplings $\pi=\mu(\d x)\pi_x(\d y)$ and the weak transport functional
\[
V_S^C(\mu,\nu)
:=
\inf_{\pi\in\Pi_S(\mu,\nu)}
\int_\R C(x,\pi_x)\,\mu(\d x),
\]
for some appropriate cost function $C:\R\times\Pc_r\to\R$. 
Our first main contribution is an approximation result in adapted Wasserstein distance: under $W_r$-convergence of marginals $(\mu^k,\nu^k)\to(\mu,\nu)$ with $\mu^k\cvd \nu^k$, every  $\pi\in\Pi_S(\mu,\nu)$ can be approximated by $\pi^k\in\Pi_S(\mu^k,\nu^k)$ such that $\mathcal{AW}_r(\pi^k,\pi)\to0$. 
As a consequence, we obtain the continuity of the functional $(\mu,\nu) \mapsto V_S^C(\mu,\nu)$, and the monotonicity principle for WSOT.
\end{abstract}

\tableofcontents

\section{Introduction}\label{sec:intro}

Over the past two decades, the \emph{constrained optimal transport}---an optimal mass transfer subject to additional structural constraints---has become a central topic. 
A prominent example is the \emph{martingale optimal transport} (MOT), motivated by the model-independent pricing and hedging in mathematical finance. 

Let $X, Y$ be Polish subspaces of $\R$, $\mu \in\mathcal P_1(X),\nu\in\mathcal P_1(Y)$ the probability measures on $X$ and $Y$ (with finite first moments), respectively, and let $\Pi(\mu,\nu)$ be a set of probability measures on $X\times Y$ (often called couplings, or transport plans) with $\mu$ and $\nu$ as (first and second) marginals: 
\[
\pi(dx,dy)=\mu(dx)\pi_x(dy)\in\Pi(\mu,\nu).
\]
Then imposing the following constraints on the conditional barycentres, give rise to the sets of martingale and supermartingale couplings of $\mu$ and $\nu$:
\begin{align}
\Pi_M(\mu,\nu)
&:=\Big\{\pi\in\Pi(\mu,\nu):
\int_Y y\,\pi_x(dy)=x \ \text{for }\mu\text{-a.e.\ }x\Big\},\label{eq:def_PiM_intro}\\
\Pi_S(\mu,\nu)
&:=\Big\{\pi\in\Pi(\mu,\nu):
\int_Y y\,\pi_x(dy)\le x \ \text{for }\mu\text{-a.e.\ }x\Big\};\label{eq:def_PiS_intro}
\end{align}
and for  a fixed cost function $c:X\times Y\to\R$, one then further seeks to minimize/maximize $\pi\mapsto\int cd\pi$ over $\Pi_{M}(\mu,\nu)$ (or $\Pi_S(\mu,\nu)$). The discrete and  continuous-time formulations of the optimal martingale transport problems were introduced by  \cite{BeiglbockHenryLaborderePenkner:13} and \cite{GalichonHenryLabordereTouzi:14}, respectively.

A basic feasibility question concerns the non-emptiness of the  sets $\Pi_M(\mu,\nu)$ and $\Pi_S(\mu,\nu)$. 
By the Strassen-type theorems (Strassen \cite{strassen1965existence}), this is characterized by stochastic orders:
\[
\Pi_M(\mu,\nu)\neq\emptyset \iff \mu\cv\nu,
\qquad
\Pi_S(\mu,\nu)\neq\emptyset \iff \mu\cvd\nu,
\]
where $\cv$ and $\cvd$ denote the \emph{convex order} and the \emph{decreasing convex order}, respectively.

Weak optimal transport (WOT) extends the classical transport problem by allowing the cost to depend on the conditional law rather than only on the locations of transported point. 
Given %probability measures on (Polish subspaces of $\mathbb R$) $X,Y$, denoted by $\mu\in\mathcal P(\mathcal X)$ and $\nu\in\mathcal P(\mathcal Y)$, and 
a measurable cost function $C:X\times\mathcal P(Y)\to\R$, the weak optimal transport problem is defined as:
\begin{equation}\label{eq:wot_problem}
V(\mu,\nu):=\inf_{\pi\in\Pi(\mu,\nu)}\int_{ X} C(x,\pi_x)\,\mu(dx).
\end{equation}
This formulation goes back to the works of Marton \cite{marton1996measure, Marton1996BoundingB} and Talagrand \cite{talagrand1995concentration, talagrand1996new}, and has since been developed extensively; see, for example,
\cite{alfonsi2017sampling, fathi2018curvature, gozlan2020mixture, shu2020hopf, shu2018hamilton, gozlan2017kantorovich, gozlan2018characterization, alibert2019new, backhoff2019existence, backhoff2020martingale, BackPam, backhoff2022applications}.

In this paper we study \emph{weak supermartingale optimal transport} (WSOT) on $X\times\mathcal P_1(Y)$, namely
\begin{equation}\label{eq:wsot_intro}
V_S^C(\mu,\nu)
:=
\inf_{\pi\in\Pi_S(\mu,\nu)}
\int_{Y} C(x,\pi_x)\,\mu(dx),
\end{equation}
which extends both classical supermartingale transport and weak martingale transport by allowing costs that depend on the full conditional law.

\subsection{Main contribution}

A central issue for both theory and applications is stability with respect to the marginals:
given $(\mu^k,\nu^k)\to(\mu,\nu)$, does one have
\[
V_S^C(\mu^k,\nu^k)\to V_S^C(\mu,\nu),
\]
and, more delicately, can the feasible sets $\Pi_S(\mu^k,\nu^k)$ approximate $\Pi_S(\mu,\nu)$ in a topology that is fine enough to control the conditional laws?

In the martingale case, stability results were initiated by Juillet \cite{Juillet:16}
and Guo and Ob{\l}\'oj \cite{GO2019}, and later established in full generality on the line
by Beiglb\"ock, Jourdain, Margheriti and Pammer
\cite{BJMP, BJMP1} using the adapted Wasserstein topology.
In higher dimensions, stability may fail even for MOT \cite{BJ.2021}.

For supermartingale transport (SOT), continuity of the value has been obtained for certain canonical constructions,
notably via shadow measures \cite{BayDengNor_1, BayDengNor_2}.
These results are quantitative but rely on specific couplings.
The present work develops a general stability theory for weak supermartingale optimal transport on the line.

For the monotonicity principle in MOT, Beiglb\"ock and Juillet \cite{BeiglbockJuillet:16}
established the cyclical monotonicity principle in the martingale setting and introduced the left-curtain coupling, which is optimal for a class of cost functions.
Motivated by ideas from MOT, Beiglb\"ock, Cox and Huesmann \cite{BeiglbockCoxHuesmann:17}
developed a geometric characterization of Skorokhod embeddings with certain optimality properties, leading to a systematic construction of optimal  embeddings. Further studies on monotonicity principles can be found in
Guo, Tan and Touzi \cite{GTT1, GTT2}.

\medskip

\begin{itemize}
    \item Assume that $\mu^k\cvd\nu^k$ for all $k$ and that $(\mu^k,\nu^k)\to(\mu,\nu)$ in $W_1$.
Our first main result (Theorem~\ref{thm:approx}) is an approximation theorem in adapted Wasserstein distance:
every $\pi\in\Pi_S(\mu,\nu)$ can be approximated by couplings $\pi^k\in\Pi_S(\mu^k,\nu^k)$ such that
$\AW(\pi^k,\pi)\to0$.
This extends \cite[Theorem~2.6]{BJMP1} from martingales to supermartingales and requires new arguments to handle
the completion of sub-probability measures under the decreasing convex order.
As a consequence, we obtain stability of the WSOT problem for a suitable class of costs,
including continuity of the optimal value and convergence of optimizers.

    \item Our second main result is a monotonicity principle for supermartingale optimal transport,
in both weak and classical formulations.
It provides a geometric characterization of the support of optimal couplings and extends the martingale monotonicity
principle of \cite{BJMP} to the supermartingale setting.
\end{itemize}

\paragraph{Organization of the paper.} Section~\ref{sec:prel} introduces notation and preliminary results.
Section~\ref{sec:main} contains the main theorems on approximation, stability, and monotonicity.
Section~\ref{sec:proof-approx} is devoted to the proof of the approximation theorem,
and Section~\ref{sec:monoto_proof} proves the monotonicity principle. Finally, 
Section~\ref{sec:appendix} collects some useful results.

\subsection{Preliminaries}\label{sec:prel}

This section fixes notation and recalls results used throughout the paper. Unless stated otherwise, we work on $\R$.

\paragraph{Measures, couplings and disintegration.}
For a Polish space $X$, let $\mathcal M(X)$ denote the set of finite, positive Borel measures on $X$, and
$\mathcal P(X)\subset\mathcal M(X)$ denote the set of probability measures.
For $r\ge 1$, let $\mathcal M_r(X)$ and $\mathcal P_r(X)$ be the subsets of measures with finite $r$-th moment.
We write $\mathcal M:=\mathcal M(\R)$, $\mathcal P:=\mathcal P(\R)$ and similarly $\mathcal M_r,\mathcal P_r$.

For $\eta\in\mathcal M_1$, define its first moment and barycentre by
\[
\overline\eta:=\int_\mathbb R x\,\eta(dx),\qquad
\ba(\eta):=\frac{\overline\eta}{\eta(\mathbb R)}
 \quad \text{if } \eta(\mathbb R)>0.
\]
We denote by $\supp(\eta)$ the closed support of $\eta$, by $\mathcal I_\eta$ the smallest interval containing $\supp(\eta)$,
and by $\ell_\eta,r_\eta$ the endpoints of $\mathcal I_\eta$ (possibly infinite).

Let $Y$ be another Polish space, for $\mu\in\mathcal M(X)$ and $\nu\in\mathcal M(Y)$ with equal mass, $\Pi(\mu,\nu)$ denotes the set of couplings on $X\times Y$. For $(\mu_k)_{k \geq 1}, \mu \in \Mc(X)$, we say $\mu_k$ converge to $\mu$ strongly, if for any measurable set $A \subset X$, $\mu_k \to \mu$ when $k \to +\infty$.

\paragraph{Wasserstein distance.}
For $\mu,\nu\in\mathcal M_r(\mathbb R)$ with $\mu(\mathbb R)=\nu(\mathbb R)$, denote by $W_r(\mu,\nu)$ their Wasserstein distance of order $r$. For $r=1$, the following duality holds, i.e.,
\begin{equation}\label{eq:Wass}
W_1(\mu,\nu)
=\sup_{\|f\|_{\Lip}\le 1}\Big\{\int f\,d\mu-\int f\,d\nu\Big\}.
\end{equation}
For $\eta\in\mathcal M(\mathbb R)$, define its distribution function $F_\eta(x):=\eta((-\infty,x])$ and the left-continuous quantile
\[
F_\eta^{-1}(u):=\inf\{x\in\mathbb R:\ F_\eta(x)\ge u\},\qquad u\in(0,\eta(\mathbb R)).
\]
Then the quantile formula yields
\begin{equation}\label{eq:WassQuantile}
W_1(\mu,\nu)
=\int_0^{\mu(\mathbb R)}\big|F_\mu^{-1}(u)-F_\nu^{-1}(u)\big|\,du
=\int_\mathbb R |F_\mu(x)-F_\nu(x)|\,dx.
\end{equation}

\paragraph{Adapted Wasserstein distance.}
Let $\pi=\mu\times\pi_x\in\mathcal M_r(\mathbb R^2)$ and $\pi'=\mu'\times\pi'_{x'}\in\mathcal M_r(\mathbb R^2)$ with $\mu(\mathbb R)=\mu'(\mathbb R)$.
The adapted Wasserstein distance is defined by
\begin{equation}\label{eq:AW_def}
\Ac\Wc_r(\pi,\pi'):=\left(\inf_{\gamma\in\Pi(\mu,\mu')}\, \int_{\mathbb R^2}\Big(|x-x'|^r+W_r(\pi_x,\pi'_{x'})^r\Big)\,\gamma(dx,dx')\right)^{1/r}.
\end{equation}
For $r=1$, with the canonical embedding
\[
J(\pi):=\mu(dx)\,\delta_{\pi_x}(dp)\in\mathcal M_1\big(\mathbb R\times\mathcal P_1\big),
\]
one has $\AW(\pi,\pi')=W_1(J(\pi),J(\pi'))$.

\paragraph{Put function and convergence in $W_1$.}
For $\eta\in\mathcal M_1$, define the put and potential functions $P_\eta, U_\eta : \R\to \R_+$ by
\[
P_\eta(x):=\int_\mathbb R (x-t)^+\,\eta(dt),\qquad
U_\eta(x):=\int_\mathbb R |x-t|\,\eta(dt).
\]
These functions are convex and determine $\eta$ via second derivatives in the sense of distributions; see
\cite{chacon1977potential, chacon1976one, kellerer73}.

\begin{lemma}\label{lem:put}
Let $(\mu^k)_{k\ge 1}\subset \mathcal P_1(\R)$ and $\mu \in \mathcal P_1(\R)$. Then
\[
W_1(\mu^k,\mu)\to0
\quad\Longleftrightarrow\quad
P_{\mu^k}\to P_\mu\ \text{ pointwise on }\mathbb R\ \text{and}\ \ba(\mu^k)\to\ba(\mu),
\]
and equivalently with uniform convergence of $P_{\mu^k}$ on $\mathbb R$. In particular,
\[
\sup_{x\in\mathbb R}\big|P_{\mu^k}(x)-P_\mu(x)\big|\le W_1(\mu^k,\mu).
\]
\end{lemma}

\paragraph{Stochastic orders and (super)martingale couplings.}
For $\eta,\chi\in\mathcal M$, we recall:
\begin{itemize}
\item $\eta\le \chi$ if $\int f\,d\eta\le\int f\,d\chi$ for all non-negative and bounded  $f$;
\item $\eta\cv\chi$ if $\eta,\chi\in\mathcal M_1$, $\eta(\mathbb R)=\chi(\mathbb R)$ and $\int f\,d\eta\le\int f\,d\chi$ for all convex $f$;
\item $\eta\cvd\chi$ if $\eta,\chi\in\mathcal M_1$,  $\eta(\mathbb R)=\chi(\mathbb R)$ and $\int f\,d\eta\le\int f\,d\chi$ for all convex nonincreasing $f$.
\end{itemize}
For measures of equal mass, one has the characterization
\[
\eta\cvd\chi \quad\Longleftrightarrow\quad P_\eta\le P_\chi \ \text{on }\mathbb R,
\]
and if additionally $\overline\eta=\overline\chi$, then $\eta\cv\chi$.
Further, given $\mu,\nu\in\mathcal M_1$ with $\mu(\mathbb R)=\nu(\mathbb R)$, recall that $\Pi_M(\mu,\nu)\subset \Pi(\mu,\nu)$
(resp. $\Pi_S(\mu,\nu)\subset \Pi(\mu,\nu)$) denotes the collection of martingale
(resp. supermartingale) couplings between $\mu$ and $\nu$.

\begin{theorem}[Strassen-type feasibility]\label{thm:strassen}
Let $\mu,\nu\in\mathcal M_1$ with $\mu(\mathbb R)=\nu(\mathbb R)$. Then
\[
\Pi_S(\mu,\nu)\neq\emptyset\ \Longleftrightarrow\ \mu\cvd\nu,
\qquad
\Pi_M(\mu,\nu)\neq\emptyset\ \Longleftrightarrow\ \mu\cv\nu.
\]
\end{theorem}

\section{Main results}\label{sec:main}

\subsection{Approximation in adapted Wasserstein topology}

Throughout, let $(\mu^k,\nu^k)_{k\ge1}\subset\Pc_r\times\Pc_r$ and $(\mu,\nu)\in\Pc_r\times\Pc_r$ satisfy
\[
W_r(\mu^k,\mu)+ W_r(\nu^k,\nu)\to0 \qquad\text{and}\qquad \mu^k\cvd \nu^k\ \text{for all }k\ge1.
\]
Without any risk of confusion, we write equally that 
$(\mu^k,\nu^k)_{k\ge1}$ converge to $(\mu,\nu)$ in $W_r$ or $(\mu^k,\nu^k) \to (\mu,\nu)$.
Then it holds $\mu\cvd\nu$ by Lemma~\ref{lem:put}.

\begin{theorem}[Approximation of supermartingale couplings]\label{thm:approx}
Let $(\mu^k,\nu^k)_{k\ge1}$ converge to $(\mu,\nu)$ in $W_r$ and assume $\mu^k\cvd\nu^k$ for all $k\ge1$.
Then for every $\pi\in\Pi_S(\mu,\nu)$,  there exists a sequence $\pi^k\in\Pi_S(\mu^k,\nu^k)$ such that
\[
\mathcal{AW}_r(\pi^k,\pi)\longrightarrow 0,\qquad k\to\infty.
\]
\end{theorem}

\begin{remark}[Context and novelty]
Theorem~\ref{thm:approx} is the supermartingale analogue of the martingale approximation theorem of \cite[Theorem~2.6]{BJMP1}.
In the martingale case, one can work on the irreducible decomposition and glue approximations without losing the barycentre constraint.
For supermartingales, an additional difficulty is that the constraint is inequality-valued; in particular, the gluing step requires a careful barycentre adjustment using free mass to the left of the compact region.

More precisely, when completing the sub-probability measures $\hat{\mu}^k$ and $\tilde{\nu}^k$
(see Section~\ref{subsec:completion-gluing}) to recover probability measures with the prescribed marginals,
one needs to prove that the put potential of the residual part of $\mu^k$ is dominated by that of the residual part of the second marginal on the whole real line.
The argument is decomposed into the comparison on an interval $J$ and on its complement $J^c$.
On $J$ the proof is straightforward.
On $J^c$, unlike in the martingale case of \cite{BJMP1}, both sides of the comparison are nontrivial, so one must return to the definitions of the auxiliary measures and establish the comparison directly.
\end{remark}

\subsection{Stability of the WSOT}
\begin{theorem}\label{thm:stability}
Fix $r\in[1,\infty)$, and let X and Y be Polish subspaces of $\R$. Let $C:X \times\sP_r(Y)\to\R$ be measurable, continuous and convex in the second argument. Assume that there exists $K>0$ such that
$$
\lvert C(x,m)\lvert\leq K\left(1+\lvert x\lvert^r+\int_{Y}\lvert y\lvert^rm(dy)\right),\qquad \forall (x,m)\in X\times\sP_r(Y).
$$
Fix $\mu\in\sP_r(X),\nu\in\sP_r(Y)$ with $\mu\cvd \nu$. Let $\mu^k\in\sP_r(X),\nu^k\in\sP_r(Y)$ satisfy $\mu^k\cvd\nu^k$, and  let $\mu^k,\nu^k$ converge to $\mu$ and $\nu$ under $W_r$, respectively. Then
\begin{equation}\label{eq:thmValueconvergence}
\lim_{k\to\infty} V^C_S(\mu^k,\nu^k) =  V^C_S(\mu,\nu)
\end{equation}
holds under one of the following conditions:
 \begin{itemize}
        \item[(A'')] $C$ is continuous;
        \item[(B'')] $\mu^k$ converges strongly to $\mu$. 
    \end{itemize}
For $k\in\N$, let $\pi^{*,k}\in\Pi_S(\mu^k,\nu^k)$ be a minimiser of $V^C_S(\mu^k,\nu^k)$. Then any accumulation point of $(\pi^{*,k})_{k\in\N}$ (with respect to the weak convergence) is a minimiser of $V^C_S(\mu,\nu)$. If $V^C_S(\mu,\nu)$ has a unique minimiser $\pi^*\in\Pi_S(\mu,\nu)$, then 
\begin{equation}\label{eq:thmCouplingconvergence}
\pi^{*,k}\xrightarrow{k\to\infty}\pi^*\quad\textrm{in } W_r,
\end{equation}
If moreover, $C$ is strictly convex in the second argument, then the convergence in \eqref{eq:thmCouplingconvergence} holds in $\mathcal{AW}_r$.
\end{theorem}

\begin{remark}\label{rem:smotCinfty}
    Let $r\geq 1$ and $X, Y$ be Polish subspaces of $\R$. For a measurable $C:X\times\sP_r(Y)\to\R\cup\{\infty\}$, consider $\tilde C:X\times\sP_r(Y)\to\R\cup\{\infty\}$ defined by $\tilde C(x,m)=C(x,m)\I_{\{\int_Y ym(dy)\leq x\}} +\infty \I_{\{\int_Y ym(dy)> x\}}$. If $C$ is convex in its second argument (respectively is lower semicontinuous in either its second argument or in both arguments, respectively satisfies $C(x,m)\geq -K\left(1+\lvert  x\lvert^r+\int_Y \lvert y\lvert ^rm(dy)\right)$), then so does $\tilde C$. Indeed, the identity map $\R\ni y\to y\in\R$ belongs to $\Phi_r(\R)=\{h:\R\to\R:h\textrm{ is continuous and }\exists\alpha>0, \forall x\in\R, \lvert h(x)\lvert\leq\alpha(1+\lvert x\lvert^r\}$, and thus $\{(x,m)\in X\times\sP_r(Y):\int_Y ym(dy)\le x\}\subset X\times\sP_r(Y)$ is closed. Also, it is easy to see that, for a fixed $x\in\R$, $\{m\in\sP_r(Y):x=\int_Y y m(dy)\}$ is convex.
\end{remark}

The proof of Theorem \ref{thm:stability} is fulfilled by the following two propositions. 

\begin{proposition}[Attainment and lower semicontinuity]\label{prop:lsc}
Fix $r\in[1,\infty)$ and let $X, Y$ be Polish subspace of $\R$. Assume that $C:X\times\sP_r(Y) \to \R$ is measurable, convex and lower semicontinous in the second argument, and there exists  $K>0$ so that 
    $$
    C(x,m)\geq -K\left(1+\lvert x\lvert^r+\int_Y\lvert y\lvert^rm(dy)\right),\qquad \forall (x,m)\in X\times\sP_r(Y).
    $$
    For $\mu\in\sP_r(X),\nu\in\sP_r(Y)$ with $\mu\cvd\nu$,  there exists a minimizer $\pi^*\in\Pi_S(\mu,\nu)$ for $V_S^C(\mu,\nu)$. If $C$ is strictly convex in the second argument and $V^C_S(\mu,\nu)$ is finite, then the minimizer is unique.

    For $k\in\N$, let $\mu^k\in\sP_r(X),\nu^k\in\sP_r(Y)$ with $\mu^k\cvd\nu^k$ converge in $W_r$ to $\mu$ and $\nu$, respectively. If one of the following holds:
    \begin{itemize}
        \item[(A)] $C$ is lower semicontinuous in both arguments;
        \item[(B)] $\mu^k$ converges strongly to $\mu$. 
    \end{itemize}
    Then $$V_S^C(\mu,\nu)\leq\liminf_{k\to\infty}V_S^C(\mu^k,\nu^k).$$
\end{proposition}
\begin{proof}
    Note that
    $$
    V_S^C(\mu,\nu)=\inf_{\pi\in\Pi_S(\mu,\nu)}\int_X C(x,\pi_x)\mu(dx)=\inf_{\pi\in\Pi(\mu,\nu)}\int_X\tilde C(x,\pi_x)\mu(dx)=V^{\tilde C}(\mu,\nu),
    $$
    where $\tilde C:X\times\sP_r(Y) \to \R$ is defined by $\tilde C(x,m)=C(x,m)\I_{\{\int_Y ym(dy)\leq x\}} +\infty \I_{\{\int_Y ym(dy)> x\}}$; recall Remark \ref{rem:smotCinfty}. Then, by following verbatim the proof of \cite[Theorem 2.4]{BJMP}, one obtains the statements concerning $V^{\tilde C}$ (and thus also $V^{C}_S$).
\end{proof}
\begin{proposition}[Upper semicontinuity]\label{prop:usc}
Fix $r\in[1,\infty)$ and let $X, Y$ be Polish subspace of $\R$. Assume that $C:X\times\sP_r(Y)\to\R$ is measurable, upper semicontinous in the second argument, and there exists  $K>0$ so that 
    $$
    C(x,m)\leq K\left(1+\lvert x\lvert^r+\int_Y\lvert y\lvert^rm(dy)\right),\qquad \forall (x,m)\in X\times\sP_r(Y).
    $$
     For $k\in\N$, let $\mu^k\in\sP_r(X),\nu^k\in\sP_r(Y)$ with $\mu^k\cvd\nu^k$ converge in $W_r$ to $\mu$ and $\nu$, respectively. If one of the following holds: 
    \begin{itemize}
        \item[(A')] $C$ is upper semicontinuous in both arguments;
        \item[(B')] $\mu^k$ converges strongly to $\mu$. 
    \end{itemize}
    Then $$\limsup_{k\to\infty}V_S^C(\mu^k,\nu^k)\leq V_S^C(\mu,\nu).$$
\end{proposition}
\begin{proof}
Fix $\pi\in\Pi_S(\mu,\nu)$. Then our approximation result, Theorem \ref{thm:approx}, ensures that there exists a sequence $\pi^k\in\Pi_S(\mu^k,\nu^k)$, $k\in\N$, which converges to $\pi$ in $\mathcal{AW}_r$. Then the result follows from the Portmanteau-type arguments; see \cite[Section A.2]{BJMP}. In particular, given the approximation result,
we can follow verbatim the proof of \cite[Theorem 2.8]{BJMP}.
\end{proof}

\begin{proof}[Proof of Theorem~\ref{thm:stability}]
Combining Proposition \ref{prop:lsc} and Proposition \ref{prop:usc}, we draw the conclusion.
\end{proof}

\subsection{Supermartingale $C$-monotonicity} \label{subsec:mono}

 A by-product of the SOT stability result, is to prove that supermartingale $C$-monotonicity is a necessary and sufficient criterion for optimality: $\pi^* \in \Pi_S(\mu,\nu)$ minimises the WSOT problem if and only if $\mu(dx) \delta_{\pi^*_x}(dp)$ is concentrated on a supermartingale $C$-monotone set. 

\begin{definition}[Supermartingale $C$-monotonicity.] \label{def:supMart_mono}
The Borel set $\Gamma \subset \R \times \Pc_1$ is supermartingale $C$-monotone if and only if there exists $M_0, M_1 \subset \R$, such that for any $N \in \N$, any collection $(x_1, p_1), \cdots, (x_N, p_N) \in \Gamma$ and $q_1, \cdots q_N \in \Pc_1$ such that $\sum_{i=1}^{N} p_i = \sum_{i=1}^{N} q_i$ and 
$$
\begin{aligned}
& \int_{\R} y p_i(dy)= \int_{\R} y q_i(dy), \ \ \forall i \in \mathcal{I}_1; \\
& \int_{\R} y p_i(dy) \geq \int_{\R} y q_i(dy), \ \ \forall i \in \mathcal{I}_0,
\end{aligned}
$$
 we have
 $$
 \sum_{i=1}^N C(x_i, p_i) \leq  \sum_{i=1}^N C(x_i, q_i).
 $$
where, $\mathcal{I}_0=\{i : x_i \in M_0 \}$, $\mathcal{I}_1=\{i : x_i \in M_1 \}$. In the following, we say a coupling $\pi \in \Pi_S(\mu, \nu)$ is Supermartingale $C$-monotone if there exists a Supermartingale $C$-monotone set $\Gamma$ with 
$$
(x, \pi_x) \in \Gamma \ \ \mbox{for } \mu(dx) \mbox{-almost every }x,
$$ 
with $M_0, M_1$ defined as the martingale region of $\pi$ to the left and right of $x^*$ (recall that $x^*$ is defined in Section \ref{sec:irreducible})   In addition, a probability measure $\P \in \Pc(\R \times \Pc_1)$ is called supermartingale $C$-monotone, if $\pi= \hat{I} (\P) \in \Pi_S(\mu,\nu)$ is supermartingale monotone.%; which is supported on a martingale $C$-monotone set, }
\end{definition}
 
 \begin{remark} \label{rmk: fini-opt}
Notice that for the  supermartingale $C$-monotonicity of the support set $\Gamma$, we do not require $\int_{\R} y p_i(dy)= \int_{\R} y q_i(dy)$, for $i \in \mathcal{I}_0$, which is consistent with \cite{NutzStebegg.18}. We then choose  $M_0$ and $M_1$ specifically as in the definition of  supermartingale $C$-monotonicity for coupling $\pi$. We emphasize here that the region of $M_1$ depends only on the marginals, while $M_0$ depends on the specific coupling.

 \end{remark}

 In the following, we shall prove that the above supermartingale $C$-monotonicity is a necessary and sufficient  condition for the optimality of the optimizer.

 We have the following main theorem:
\begin{theorem} \label{thm:monotonicity}
 Let $\mu\cvd \nu \in \Pc_r(\R)$ with $r\geq 1$. Assume that  $C:\R \times\Pc_r(\R) \to \R$ is  measurable, continuous, and convex in the second argument and that there exists  $K>0$ so that 
$$
\forall (x,p) \in \R \times \Pc_r(\R), \ |C(x,p)| \leq K \left( 1 + |x|^r  + \int_{\R} |y|^r  p (dy) \right).
$$
Then  $\pi \in \Pi_S(\mu, \nu)$ is supermartingale $C$-monotone w.r.t. some $\Gamma$ if and only if $\pi$ is optimal for (WSOT).
\end{theorem}
\begin{proof}
Combine Proposition \ref{prop:sufficiency} and Proposition \ref{prop:necessity}.
\end{proof}

Using the above main theorem, we can obtain the monotonicity principle of classical SOT. Recall that the definition of finite optimality was given in \ref{subsec:def_mono}.

\begin{corollary}(Monotonicity principle for SOT). Let \( r \geq 1 \), \( \mu\cvd \nu \in \mathcal{P}_r \), $c: \R \times \R \to \R$ be measurable and such that $y \mapsto c(x,y)$ is continuous for all \( x \in \mathbb{R} \) and \( \sup_{(x,y) \in \mathbb{R}^2} \frac{|c(x,y)|}{|1+|x|^r + |y|^r} < \infty \). Then \( \pi \in \Pi_S(\mu, \nu) \) is optimal for (SOT) if and only if it is concentrated on a finitely optimal set.
\end{corollary}

\begin{proof}
The argument is similar as \cite[Corollary 3.4]{BJMP}, and we omit it here.
\end{proof}

\section{Proof of Theorem \ref{thm:approx}}\label{sec:proof-approx}

This section is devoted to the proof of Theorem  \ref{thm:approx}.  In view of \cite[Lemma 5.1]{BJMP1}, it suffices to prove Theorem \ref{thm:approx} in the case of $r=1$.
Throughout Section \ref{sec:proof-approx},   we fix a sequence $(\mu^k,\nu^k)_{k\ge1}\subset\Pc_1\times\Pc_1$ satisfying
\begin{equation}\label{eq:standing-marginals}
W_1(\mu^k,\mu)+W_1(\nu^k,\nu)\longrightarrow 0\qquad \mbox{and} \qquad \mu^k\cvd\nu^k\ \text{ for every }k\ge1,
\end{equation}
and we pick a  supermartingale coupling $\pi=\mu(dx)\pi_x(dy)\in\Pi_S(\mu,\nu)$.
We aim to construct a sequence $(\pi^k)_{k\ge1}\subset\Pc_1(\mathbb R^2)$ such that  
\[
\mathcal{AW}_1(\pi^k,\pi)\longrightarrow 0 
 \qquad \mbox{and} \qquad \pi^k\in\Pi_S(\mu^k,\nu^k)\ \text{ for every } k\ge1.
\]
Recall that for a finite measure $\eta\in\Mc_1$, we write
\[
\bar\eta:=\int_{\R}x\,\eta(dx),
\qquad
P_\eta(x):=\int_{\R}(x-y)^+\eta(dy),
\]
and, for two finite measures of the same mass,
\[
\Delta(\eta,\chi):=\bar\eta-\bar\chi.
\]
Recall that $\eta\cvd\chi$ is equivalent to $P_\eta\le P_\chi$ on $\R$, and that
$\eta\cv\chi$ is equivalent to $\eta\cvd\chi$ together with
$\bar\eta=\bar\chi$.

The proof has three layers.  First, we reduce the problem to the only component where a
new argument is needed: the strict, genuinely supermartingale irreducible component.
Second, on this component we construct, for every small $\varepsilon>0$, a coupling
\[
\bar\pi^{k,\varepsilon}
\in
\Pi_S\big(\mu^k,\lambda^{k,\varepsilon}\big),
\qquad
\lambda^{k,\varepsilon}:=
\varepsilon\nu^k+(1-\varepsilon)\nu^{R,\alpha,k},
\]
where 
\[
\nu^{R,\alpha,k}\cvd\nu^k,\qquad \mathcal{AW}_1( \bar\pi^{k,\varepsilon} , \pi) \longrightarrow 0  \mbox{ if first } k\to\infty \mbox{ and then } \varepsilon\downarrow 0.
\]
Third, we use a final
supermartingale step from $\lambda^{k,\varepsilon}$ to $\nu^k$ and choose
$\varepsilon$ diagonally.

For reference, the main objects in the strict component are the following.  All of them
are introduced precisely below.
\begin{center}
\begin{tabular}{ll}
\toprule
symbol & meaning \\
\midrule
$I=(\ell,+\infty)$ & irreducible supermartingale interval; $\mu(I)=1$ \\
$\Delta=\bar\mu-\bar\nu>0$ & total supermartingale defect \\
$K\subset(a,b)\Subset I$ & compact set carrying almost all relevant mass \\
$\pi^R$ & kernelwise compactification of $\pi$ \\
$\pi^{R,\alpha}$ & affine contraction of $\pi^R$ toward the diagonal \\
$\nu^{R,\alpha}$ & second marginal of $\pi^{R,\alpha}$ \\
$\nu^{R,\alpha,k}$ & intermediate target, with $\mu^k\cvd\nu^{R,\alpha,k}\cvd\nu^k$ \\
$L\Subset I$ & compact set containing the $y$-support of $\pi^{R,\alpha}|_{K\times\R}$ \\
$\hat\pi^k$ & localisation of an ordinary coupling from $\mu^k$ to $\nu^{R,\alpha,k}$ \\
$\nu_-^k$ & free mass of $\nu^{R,\alpha,k}$ lying strictly to the left of $K$ \\
$\tilde\pi^k$ & barycentre-corrected version of $\hat\pi^k$ \\
$\mu^k_{\rm rem},\nu^k_{\rm rem}$ & residual marginals after inserting $(1-2\varepsilon)\tilde\pi^k$ \\
\bottomrule
\end{tabular}
\end{center}

\subsection{Reduction to irreducible components}\label{subsec:reduction-irreducible}

Let us recall the irreducible decomposition of supermartingale couplings from
Lemma~\ref{lem:irreducible}.  Define 
\[
D:=P_\nu-P_\mu 
\qquad \mbox{and} \qquad 
x^*:=\sup\{x\in\R:\ D(x)=0\}.
\]
It follows immediately that $D\ge 0$ and 
\[
D(x)\sim 0 \mbox{ as } x\to -\infty \qquad \mbox{and} \qquad D(x)\sim \overline\mu - \overline\nu  \mbox{ as } x\to +\infty.  
\]
Set
\[
I_0:=(x^*,+\infty),
\]
and let $(I_n)_{n\ge1}$,  $I_n:=(l_n,r_n)$,   be the open connected components of
$\{D>0\}\cap(-\infty,x^*)$.  Define
\[
I_{-1}:=\R\setminus\bigcup_{n\ge0}I_n.
\]
Then
\begin{equation}\label{eq:limit-component-decomposition}
\mu=\sum_{n\ge-1}\mu_n,
\qquad
\nu=\sum_{n\ge-1}\nu_n,
\qquad
\pi=\sum_{n\ge-1}\pi_n,
\end{equation}
where
\[
\mu_{-1}=\nu_{-1}=:\eta,
\quad
\pi_{-1}=\eta(dx)\delta_x(dy),\quad 
\pi_0\in\Pi_S(\mu_0,\nu_0),
\quad
\pi_n\in\Pi_M(\mu_n,\nu_n) \text{ for all } n\ge1.
\]
In particular,  $\mu_{-1}=\nu_{-1}$,  $\mu_0\cvd\nu_0$ and $\mu_n\cv\nu_n$ for all $n\ge1$.

The next proposition is the supermartingale counterpart of the componentwise splitting
used in the martingale proof.  The construction follows the limiting quantile blocks of
$\mu$ and then transports these blocks with an arbitrary supermartingale coupling between
$\mu^k$ and $\nu^k$.

\begin{proposition}\label{prop:decompo1}
For every $k$,  there exist finite measures
$(\mu_n^k,\nu_n^k)_{n\ge-1}\subset \mathcal M_1\times \mathcal M_1$ such that
\begin{equation}\label{eq:component-decomp-k}
\mu^k=\sum_{n\ge-1}\mu_n^k,
\qquad
\nu^k=\sum_{n\ge-1}\nu_n^k,
\qquad
\mu_n^k\cvd\nu_n^k\quad\text{for every }n\ge-1.
\end{equation}
Furthermore, for every $n\ge-1$,
\begin{equation}\label{eq:component-conv-k}
W_1(\mu_n^k,\mu_n)+W_1(\nu_n^k,\nu_n)\longrightarrow0.
\end{equation}
\end{proposition}

\begin{proof}[Proof of Proposition~\ref{prop:decompo1}]
Let
\(Q:=F_\mu^{-1}
\) and \(Q_k:=F_{\mu^k}^{-1}.
\)
For each component \(I_n\), \(n\geq -1\), set
\[
        A_n:=\{u\in(0,1):Q(u)\in I_n\}.
\]
Then the sets \((A_n)_{n\geq -1}\) form a measurable partition of \((0,1)\), up to
Lebesgue-null sets, and
\[
        \mu_n=Q_\#\big(\mathcal L^1|_{A_n}\big).
\]
For every \(k\), choose an arbitrary coupling
\(
        \Gamma^k\in\Pi_S(\mu^k,\nu^k),
\)
and write its disintegration as
\[
        \Gamma^k(dx,dy)=\mu^k(dx)\Gamma_x^k(dy).
\]
For \(n\geq -1\), define \(\gamma_n^k\in \mathcal M_1(\mathbb R^2)\) by
\[
        \gamma_n^k(dx,dy)
        :=
        \int_{A_n}\delta_{Q_k(u)}(dx)\Gamma^k_{Q_k(u)}(dy)\,du.
\]
Let \(\mu_n^k\) and \(\nu_n^k\) denote respectively the first and second marginals of
\(\gamma_n^k\). Then
\[
        \mu_n^k=(Q_k)_\#\big(\mathcal L^1|_{A_n}\big).
\]
Since the sets \(A_n\) partition \((0,1)\), we have
\[
        \mu^k=\sum_{n\geq -1}\mu_n^k,
        \qquad
        \nu^k=\sum_{n\geq -1}\nu_n^k,
        \qquad
        \Gamma^k=\sum_{n\geq -1}\gamma_n^k.
\]
Moreover, for every non-negative Borel function \(f\),
\[
        \int f\,d\gamma_n^k
        =
        \int_{A_n}\int f(Q_k(u),y)\Gamma^k_{Q_k(u)}(dy)\,du
        \leq
        \int_0^1\int f(Q_k(u),y)\Gamma^k_{Q_k(u)}(dy)\,du
        =
        \int f\,d\Gamma^k.
\]
Hence
\(
        \gamma_n^k\leq \Gamma^k.
\)
In particular, \(\mu_n^k\leq \mu^k\) and \(\nu_n^k\leq\nu^k\). Since
\(\gamma_n^k\) has the same conditional kernel as \(\Gamma^k\) on its first marginal, and
\(\Gamma^k\) is a supermartingale coupling, we obtain
\(
        \gamma_n^k\in\Pi_S(\mu_n^k,\nu_n^k).
\)
Equivalently,
\(
        \mu_n^k\cvd\nu_n^k.
\)

We next prove the convergence of the first marginals. The quantile coupling between
\(\mu_n^k\) and \(\mu_n\) gives
\[
        W_1(\mu_n^k,\mu_n)
        \leq
        \int_{A_n}|Q_k(u)-Q(u)|\,du
        \leq
        \int_0^1|Q_k(u)-Q(u)|\,du
        =
        W_1(\mu^k,\mu).
\]
Therefore
\(
        W_1(\mu_n^k,\mu_n)\longrightarrow 0.
\)

It remains to identify the limit of the second marginals. Fix \(n\geq -1\). We claim that
\[
        W_1(\nu_n^k,\nu_n)\longrightarrow 0.
\]
It is enough to prove that every subsequential limit is \(\nu_n\). Since
\(
        \gamma_n^k\leq \Gamma^k,
\)
and since the marginals \((\mu^k,\nu^k)\) converge in \(W_1\), the families
\((\Gamma^k)_k\) and \((\gamma_n^k)_k\) are tight and uniformly integrable. Thus, after
passing to a subsequence, we may assume that
\[
        \Gamma^k\longrightarrow \Gamma,
        \qquad
        \gamma_n^k\longrightarrow \widetilde\Gamma_n
\]
weakly, and even in \(W_1\). The limit \(\Gamma\) belongs to \(\Pi_S(\mu,\nu)\), because the
supermartingale property is closed under \(W_1\)-convergence.

Since \(\gamma_n^k\leq\Gamma^k\), passing to the limit against non-negative bounded
continuous test functions gives
\(
        \widetilde\Gamma_n\leq \Gamma.
\)
The first marginal of \(\widetilde\Gamma_n\) is the \(W_1\)-limit of \(\mu_n^k\), hence is
\(\mu_n=\mu|_{I_n}\). We now show that
\[
        \widetilde\Gamma_n=\Gamma|_{I_n\times\mathbb R}.
\]
Indeed, set
\[
        \Delta_n:=\Gamma-\widetilde\Gamma_n.
\]
Then \(\Delta_n\) is a non-negative finite measure. For every Borel set
\(B\subset I_n\),
\[
        \Delta_n(B\times\mathbb R)
        =
        \Gamma(B\times\mathbb R)-\widetilde\Gamma_n(B\times\mathbb R)
        =
        \mu(B)-\mu_n(B)
        =
        0.
\]
Thus \(\Delta_n\) has no mass over \(I_n\). On the other hand,
\[
        \widetilde\Gamma_n(I_n^c\times\mathbb R)
        =
        \mu_n(I_n^c)
        =
        0.
\]
Therefore \(\widetilde\Gamma_n\) is exactly the restriction of \(\Gamma\) to
\(I_n\times\mathbb R\):
\[
        \widetilde\Gamma_n=\Gamma|_{I_n\times\mathbb R}.
\]
By the irreducible decomposition of supermartingale couplings,   the second marginal of
\(\Gamma|_{I_n\times\mathbb R}\) is precisely \(\nu_n\). Hence the second marginal of
\(\widetilde\Gamma_n\) is \(\nu_n\). Consequently every weak subsequential limit of
\((\nu_n^k)_k\) is \(\nu_n\).

Finally, since \(\nu_n^k\leq\nu^k\) and \(\nu^k\to\nu\) in \(W_1\), the family
\((\nu_n^k)_k\) is uniformly integrable:
\[
        \lim_{M\to\infty}\sup_k
        \int_{\{|y|>M\}}|y|\,\nu_n^k(dy)
        =
        0.
\]
Thus weak convergence of \(\nu_n^k\) to \(\nu_n\) upgrades to \(W_1\)-convergence. Therefore
\[
        W_1(\nu_n^k,\nu_n)\longrightarrow 0.
\]
Combining this with the convergence of \(\mu_n^k\) proves the proposition.
\end{proof}

We shall also use the following two elementary but useful facts.

\begin{lemma}\label{lem:diagonal-component}
Let $\eta^k,v^k,\eta\in\Mc_1$ have the same total mass and satisfy
\[
W_1(\eta^k,\eta)+W_1(v^k,\eta)\longrightarrow0.
\]
Let $\chi:=\eta(dx)\delta_x(dy)$.  If $\eta^k\cvd v^k$ and
$\chi^k\in\Pi_S(\eta^k,v^k)$, then
\[
\mathcal{AW}_1(\chi^k,\chi)\longrightarrow0.
\]
In particular,
\[
\int_{\R^2}|x-y|\,\chi^k(dx,dy)\longrightarrow0.
\]
\end{lemma}

\begin{proof}
The marginals of $(\chi^k)_k$ converge to $\eta$, hence the sequence is tight and has
uniformly integrable first moments.  Let $\chi'$ be any weak limit.  Then
$\chi'\in\Pi_S(\eta,\eta)$.  Therefore $\chi'=\chi$.  The convergence is in
$W_1$ by uniform integrability.

To obtain adapted convergence, let $\tilde\chi^k\in\Pi(\eta^k,v^k)$ be any sequence with
$\mathcal{AW}_1(\tilde\chi^k,\chi)\to0$, whose existence follows from the ordinary
stability of couplings.  Since both $\chi^k$ and $\tilde\chi^k$ converge to the diagonal
in $W_1$,
\[
\mathcal{AW}_1(\chi^k,\tilde\chi^k)
\le
\int |x-y|\,(\chi^k+\tilde\chi^k)(dx,dy)
\longrightarrow0.
\]
The triangle inequality yields $\mathcal{AW}_1(\chi^k,\chi)\to0$.
\end{proof}

\begin{lemma}\label{lem:convex-lift}
Let \(\eta,\chi\in\Mc_1\) have the same total mass and assume
\(\eta\cvd\chi\). Set
\(
        \delta:=\bar\eta-\bar\chi\ge0 .
\)
Define \(\chi^\sharp\in\Mc_1\) by
\begin{equation}\label{eq:convex-lift-def}
        P_{\chi^\sharp}:=P_\eta\vee(P_\chi-\delta).
\end{equation}
Then
\[
        \eta\cv\chi^\sharp\cvd\chi .
\]
Moreover, if \(\eta^k\cvd\chi^k\), \(\eta^k\to\eta\),
\(\chi^k\to\chi\) in \(W_1\), and \(\bar\eta=\bar\chi\), then the corresponding
measures \((\chi^k)^\sharp\) converge to \(\chi\) in \(W_1\).
\end{lemma}

\begin{proof}
Since \(\eta\cvd\chi\), we have \(P_\eta\le P_\chi\).  The function
\(
        P_\eta\vee(P_\chi-\delta)
\)
is convex, nondecreasing, vanishes at \(-\infty\), and has the same right asymptote as
\(P_\eta\). 
Hence \(P_\eta\vee(P_\chi-\delta)\) is the put function of a measure
\(\chi^\sharp\) with the same mass and the same first moment as \(\eta\).

Since
\(
        P_\eta\le P_{\chi^\sharp}
\)
and \(\bar\eta=\overline{\chi^\sharp}\), we obtain
\(
        \eta\cv\chi^\sharp .
\) 
On the other hand, using \(P_\eta\le P_\chi\) and \(P_\chi-\delta\le P_\chi\), we get
\[
        P_{\chi^\sharp}
        =
        P_\eta\vee(P_\chi-\delta)
        \le P_\chi .
\]
Therefore
\(
        \chi^\sharp\cvd\chi .
\)

It remains to prove the convergence statement. Let
\(
        \delta_k:=\bar\eta^k-\bar\chi^k\ge0.
\)
Since \(\eta^k\to\eta\) and \(\chi^k\to\chi\) in \(W_1\), and since
\(\bar\eta=\bar\chi\), we have
\(
        \delta_k\longrightarrow0.
\)
Moreover,
\(
        P_{(\chi^k)^\sharp}
        =
        P_{\eta^k}\vee(P_{\chi^k}-\delta_k).
\)
Because \(P_{\eta^k}\le P_{\chi^k}\), it follows that
\[
        P_{\chi^k}-\delta_k
        \le
        P_{(\chi^k)^\sharp}
        \le
        P_{\chi^k}.
\]
Thus \(P_{(\chi^k)^\sharp}\to P_\chi\) uniformly on \(\mathbb R\), since
\(P_{\chi^k}\to P_\chi\) uniformly and \(\delta_k\to0\). Also
\[
        \overline{(\chi^k)^\sharp}=\bar\eta^k\longrightarrow\bar\eta=\bar\chi.
\]
By the characterization of \(W_1\)-convergence through put functions and barycentres,
\[
        W_1\big((\chi^k)^\sharp,\chi\big)\longrightarrow0.
\]
\end{proof}

We now complete the reduction.  Apply Proposition~\ref{prop:decompo1}.  On the diagonal
component $n=-1$, for any
$\chi^k\in\Pi_S(\mu_{-1}^k,\nu_{-1}^k)$, Lemma~\ref{lem:diagonal-component} gives
\[
\mathcal{AW}_1(\chi^k,\pi_{-1})\longrightarrow0.
\]
Next consider a component $n\ge 1$ with zero limiting barycentre defect,
$\bar\mu_n=\bar\nu_n$.  Then $\pi_n$ is in fact a martingale coupling.  Let
$(\nu_n^k)^\sharp$ be the convex lift of $\nu_n^k$ over $\mu_n^k$ given by
Lemma~\ref{lem:convex-lift}.  Then
\[
\mu_n^k\cv(\nu_n^k)^\sharp,
\qquad
(\nu_n^k)^\sharp\cvd\nu_n^k,
\qquad
W_1((\nu_n^k)^\sharp,\nu_n)\longrightarrow0.
\]
By the martingale approximation theorem on the line, \cite[Theorem~2.6]{BJMP1}, applied
after normalising the finite measures when necessary, there exist
\[
\hat\pi_n^k\in\Pi_M(\mu_n^k,(\nu_n^k)^\sharp)
\quad\text{such that}\quad
\mathcal{AW}_1(\hat\pi_n^k,\pi_n)\longrightarrow0.
\]
Choose $M_n^k\in\Pi_S((\nu_n^k)^\sharp,\nu_n^k)$ and compose $\hat\pi_n^k$ with
$M_n^k$ on the second coordinate.  Since both marginals of $M_n^k$ converge to $\nu_n$,
Lemma~\ref{lem:diagonal-component} gives
\[
\int_{\R^2}|z-y|\,M_n^k(dz,dy)\longrightarrow0.
\]
The composed coupling, denoted $\pi_n^k$, belongs to $\Pi_S(\mu_n^k,\nu_n^k)$ and
satisfies
\[
\mathcal{AW}_1(\pi_n^k,\pi_n)\longrightarrow0.
\]
Thus the only component not covered by the martingale theorem is $n=0$ when
\[
\mu_0(\R)>0,
\qquad
\bar\mu_0-\bar\nu_0>0.
\]
The rest of this section proves the approximation result for this strict component.  Once it
is proved, the finite partial sums converge by \cite[Lemma~3.7]{BJMP1},  where the passage from
finite partial sums to the full sum follows from \cite[Lemma~3.6(b)]{BJMP1} and the uniform
integrability of the marginals, exactly as in the martingale proof.  Hence Theorem
\ref{thm:approx} follows from the strict irreducible case.

\subsection{Auxiliary operations and estimates}\label{subsec:lemmas}

We now collect the operations used in the proof of the strict component.  For
$\eta,\chi\in\Pc_1$, define
\begin{equation}\label{eq:cd-lattice-def}
P_{\eta\vee_{cd}\chi}:=P_\eta\vee P_\chi,
\qquad
P_{\eta\wedge_{cd}\chi}:=(P_\eta\wedge P_\chi)^c,
\end{equation}
where $h^c$ denotes the largest convex minorant of $h$.  Then
\[
\ba(\eta\vee_{cd}\chi)=\ba(\eta)\wedge\ba(\chi),
\qquad
\ba(\eta\wedge_{cd}\chi)=\ba(\eta)\vee\ba(\chi).
\]

\begin{lemma}\label{lem:truncate}
Let $\pi=\mu(dx)\pi_x(dy)\in\Pi_S(\mu,\nu)$ with $\mu,\nu\in\Pc_1$.  For $R>0$, define
\[
m_x:=\int y\,\pi_x(dy),
\]
and
\begin{equation}\label{eq:truncate-kernel}
\pi_x^R:=
\begin{cases}
\displaystyle
\pi_x\wedge_{cd}
\left(\frac{R-m_x}{2R}\delta_{-R}+\frac{R+m_x}{2R}\delta_R\right),
& |m_x|\le R,\\[2mm]
\delta_{m_x},& |m_x|>R.
\end{cases}
\end{equation}
Let $\pi^R:=\mu(dx)\pi_x^R(dy)$ and denote its second marginal by $\nu^R$.  Then
\[
\pi^R\in\Pi_S(\mu,\nu^R),
\qquad
\nu^R\cv\nu,
\qquad
\mathcal{AW}_1(\pi^R,\pi)\longrightarrow0
\quad\text{as }R\to\infty.
\]
\end{lemma}

\begin{proof}
For $|m_x|\le R$, the two-point law in \eqref{eq:truncate-kernel} has barycentre
$m_x$.  Hence $\pi_x^R\cvd\pi_x$ and $\pi_x^R$ has the same barycentre as $\pi_x$, so
$\pi_x^R\cv\pi_x$.  For $|m_x|>R$, Jensen's inequality gives
$\delta_{m_x}\cv\pi_x$.  Therefore $\nu^R\cv\nu$, and
\[
\int y\,\pi_x^R(dy)=m_x\le x,
\]
so $\pi^R$ is a supermartingale coupling.

For fixed $x$, $W_1(\pi_x^R,\pi_x)\to0$ as $R\to\infty$; this is the usual compactification
of a probability measure by its convex-order infimum with a two-point law of the same
barycentre.  Moreover,
\[
W_1(\pi_x^R,\pi_x)
\le
\int |y|\,\pi_x^R(dy)+\int |y|\,\pi_x(dy)
\le
2\int |y|\,\pi_x(dy),
\]
where the last inequality uses $\pi_x^R\cv\pi_x$.  Dominated convergence gives
\[
\mathcal{AW}_1(\pi^R,\pi)
\le
\int W_1(\pi_x^R,\pi_x)\,\mu(dx)
\longrightarrow0.
\]
\end{proof}

\begin{lemma}\label{lem:contract}
Let $\pi=\mu(dx)\pi_x(dy)\in\Pi_S(\mu,\nu)$ and let $\alpha\in(0,1)$.  Set
\[
T_{x,\alpha}(y):=\alpha y+(1-\alpha)x,
\qquad
\pi_x^\alpha:=(T_{x,\alpha})_\#\pi_x,
\qquad
\pi^\alpha:=\mu(dx)\pi_x^\alpha(dy),
\]
and denote the second marginal of $\pi^\alpha$ by $\nu^\alpha$.  Then
$\pi^\alpha\in\Pi_S(\mu,\nu^\alpha)$ and
\[
\mathcal{AW}_1(\pi^\alpha,\pi)
\le
(1-\alpha)\left(\int |x|\,\mu(dx)+\int |y|\,\nu(dy)\right).
\]
\end{lemma}

\begin{proof}
The conditional barycentre of $\pi_x^\alpha$ is
$\alpha\int y\pi_x(dy)+(1-\alpha)x\le x$.  The displayed estimate follows by coupling
$y$ with $T_{x,\alpha}(y)$ and using
$|T_{x,\alpha}(y)-y|=(1-\alpha)|x-y|$.
\end{proof}

\begin{lemma}\label{lem:put_minmax}
Let $\eta^k\to\eta$ and $\chi^k\to\chi$ in $W_1$.  Then
\[
W_1(\eta^k\vee_{cd}\chi^k,\eta\vee_{cd}\chi)\longrightarrow0,
\qquad
W_1(\eta^k\wedge_{cd}\chi^k,\eta\wedge_{cd}\chi)\longrightarrow0.
\]
\end{lemma}

\begin{proof}
By Lemma~\ref{lem:put}, $W_1$-convergence is equivalent to uniform convergence of put
functions together with convergence of barycentres.  The first assertion follows from
\[
P_{\eta^k\vee_{cd}\chi^k}=P_{\eta^k}\vee P_{\chi^k}
\longrightarrow
P_\eta\vee P_\chi=P_{\eta\vee_{cd}\chi}
\]
uniformly.  For the infimum, uniform convergence of
$P_{\eta^k}\wedge P_{\chi^k}$ to $P_\eta\wedge P_\chi$ and the variational
characterisation of the convex minorant give, for every $\delta>0$ and all large $k$,
\[
(P_\eta\wedge P_\chi)^c-\delta
\le
(P_{\eta^k}\wedge P_{\chi^k})^c
\le
(P_\eta\wedge P_\chi)^c+\delta.
\]
Thus the put functions converge uniformly.  The barycentre formula following
\eqref{eq:cd-lattice-def} gives the convergence of barycentres, and Lemma~\ref{lem:put}
concludes the proof.
\end{proof}

\begin{lemma}\label{lem:defect-vanishes}
Let $\gamma^k=\eta^k(dx)\gamma_x^k(dy)$ and
$\gamma=\eta(dx)\gamma_x(dy)$ be finite couplings with equal total masses.  Assume
$\gamma\in\Pi_S(\eta,\chi)$ and
$\mathcal{AW}_1(\gamma^k,\gamma)\to0$.  Then
\[
\int\left(\int y\,\gamma_x^k(dy)-x\right)^+\eta^k(dx)
\longrightarrow0.
\]
\end{lemma}

\begin{proof}
Let $\zeta^k\in\Pi(\eta^k,\eta)$ be an optimizer in the definition of
$\mathcal{AW}_1(\gamma^k,\gamma)$.  Since
$\int y\gamma_{x'}(dy)\le x'$ for $\eta$-a.e. $x'$,
\begin{align*}
&\int\left(\int y\,\gamma_x^k(dy)-x\right)^+\eta^k(dx) \\
&\quad\le
\int\left(
\left|x-x'\right|+
W_1(\gamma_x^k,\gamma_{x'})
\right)\zeta^k(dx,dx')
\le
\mathcal{AW}_1(\gamma^k,\gamma).
\end{align*}
\end{proof}

\subsection{The strict irreducible component}\label{subsec:strict-component}

From now until the final diagonal argument, we work on one strict irreducible component.
After normalising its mass to one, we assume
\begin{equation}\label{eq:strict-standing}
\pi\in\Pi_S(\mu,\nu),
\qquad
I:=\{x\in\R:P_\mu(x)<P_\nu(x)\}=(\ell,+\infty),
\qquad
\mu(I)=1,
\end{equation}
and
\begin{equation}\label{eq:genuine-defect}
\Delta:=\bar\mu-\bar\nu>0.
\end{equation}
Fix $\varepsilon\in(0,1/4)$.  All objects constructed in the next subsections depend on
$\varepsilon$, but this dependence is suppressed until the final diagonal step.

Write
\[
m_x:=\int y\,\pi_x(dy).
\]
Choose numbers
\[
\ell<\tilde a<a<b<+\infty
\]
and a compact set $K\subset(a,b)$ such that
\begin{equation}\label{eq:K-choice}
\mu(K^c)+
\int_{K^c}\left(|x|+\int |y|\,\pi_x(dy)\right)\mu(dx)
\le\varepsilon,
\end{equation}
and choose $R_0>0$ such that
\begin{equation}\label{eq:R0-choice}
|m_x|\le R_0
\quad\text{for }\mu\text{-a.e. }x\in K.
\end{equation}
The point $\tilde a$ is chosen so that $P_\nu(\tilde a)>P_\mu(\tilde a)$, hence
$\nu(( -\infty,\tilde a))>0$.

Apply Lemma~\ref{lem:truncate} to obtain $\pi^R$ and then Lemma~\ref{lem:contract} to
$\pi^R$.  Thus
\begin{equation}\label{eq:piRA-def}
\pi^{R,\alpha}:=\mu(dx)\pi_x^{R,\alpha}(dy),
\qquad
\pi_x^{R,\alpha}:=(T_{x,\alpha})_\#\pi_x^R,
\end{equation}
where $T_{x,\alpha}(y):=\alpha y+(1-\alpha)x$.  Let $\nu^{R,\alpha}$ be the second
marginal of $\pi^{R,\alpha}$.

\begin{proposition}\label{prop:pi_R_alpha}
There exist $R>R_0\vee |a|\vee |b|$ and $\alpha\in(0,1)$ such that
\begin{equation}\label{eq:RA-close}
\mathcal{AW}_1\big(\varepsilon\pi+(1-\varepsilon)\pi^{R,\alpha},\pi\big)
<\varepsilon,
\end{equation}
\begin{equation}\label{eq:alpha-choice}
\alpha>\frac{2R-a-\tilde a}{2R-2\tilde a},
\end{equation}
and the following properties hold.

\begin{enumerate}[(i)]
\item There is a compact interval $L\Subset I$ such that
$\mu|_K(dx)\pi_x^{R,\alpha}(dy)$ is concentrated on $K\times L^\circ$.

\item The measure $\nu^{R,\alpha}$ has positive mass strictly to the left of $K$:
\begin{equation}\label{eq:left-free-mass-limit}
\nu^{R,\alpha}\left(\left(\ell,\frac{a+\tilde a}{2}\right)\right)>0.
\end{equation}

\item We have
\begin{equation}\label{eq:RA-order}
\pi^{R,\alpha}\in\Pi_S(\mu,\nu^{R,\alpha}),
\qquad
\nu^{R,\alpha}\cvd\nu.
\end{equation}
Moreover,
\begin{equation}\label{eq:RA-defect}
\Delta^{R,\alpha}:=\Delta(\mu,\nu^{R,\alpha})
=\alpha\Delta>0.
\end{equation}
\end{enumerate}
\end{proposition}

\begin{proof}
By Lemmas~\ref{lem:truncate} and \ref{lem:contract}, $\pi^{R,\alpha}$ can be made
arbitrarily close to $\pi$ in $\mathcal{AW}_1$ by first taking $R$ large and then taking
$\alpha$ close to one.  This gives \eqref{eq:RA-close}, and we may impose
\eqref{eq:alpha-choice} at the same time.

Since $P_\nu(\tilde a)>P_\mu(\tilde a)\ge0$, we have
$\nu(( -\infty,\tilde a))>0$.  As $\nu^R\to\nu$ in $W_1$, for large $R$,
$\nu^R(( -\infty,\tilde a))>0$.  Note
\[
\nu^R(( -\infty,\tilde a)) = \int_{x\ge R}\pi_x^R(( -\infty,\tilde a))\,\mu(dx)+ \int_{x<R}\pi_x^R(( -\infty,\tilde a))\,\mu(dx).
\]
So,  for large $R$,  
\[
 \int_{x<R}\pi_x^R(( -\infty,\tilde a))\,\mu(dx)>0.
\]
If $x<R$ and $y<\tilde a$, then \eqref{eq:alpha-choice} gives
\[
\alpha y+(1-\alpha)x
<
\alpha\tilde a+(1-\alpha)R
<
\frac{a+\tilde a}{2}.
\]
This proves \eqref{eq:left-free-mass-limit}.

For $x\in K$, \eqref{eq:R0-choice} and $R>R_0$ imply that $\pi_x^R$ is supported on
$[-R,R]$.  Hence $\pi_x^{R,\alpha}$ is supported on the compact interval
\[
\big[\alpha(-R)+(1-\alpha)a,
       \alpha R+(1-\alpha)b\big]\cap I.
\]
Choosing a compact interval $L\Subset I$ whose interior contains this interval gives the
localisation statement.

The supermartingale property follows from
\[
\int y\,\pi_x^{R,\alpha}(dy)
=
\alpha m_x+(1-\alpha)x
\le x.
\]
To prove $\nu^{R,\alpha}\cvd\nu$, let $\varphi$ be convex and nonincreasing.  Since
$\pi_x^R\cv\pi_x$ and $m_x\le x$,
\begin{align*}
\int\varphi(y)\,\pi_x^{R,\alpha}(dy)
&=\int\varphi(\alpha y+(1-\alpha)x)\,\pi_x^R(dy)\\
&\le
\alpha\int\varphi(y)\,\pi_x^R(dy)+(1-\alpha)\varphi(x)\\
&\le
\alpha\int\varphi(y)\,\pi_x(dy)+(1-\alpha)\int\varphi(y)\,\pi_x(dy).
\end{align*}
In the last step we used
$\varphi(x)\le\varphi(m_x)\le\int\varphi(y)\,\pi_x(dy)$.  Integrating in $x$ gives
$\nu^{R,\alpha}\cvd\nu$.
Finally, Lemma~\ref{lem:truncate} preserves the conditional barycentre $m_x$, hence
\[
\bar\nu^{R,\alpha}
=
\int\big(\alpha m_x+(1-\alpha)x\big)\mu(dx)
=
\alpha\bar\nu+(1-\alpha)\bar\mu,
\]
which gives \eqref{eq:RA-defect}.
\end{proof}

\subsection{Localising the approximating sequence}\label{subsec:localise-correct}

We now construct the intermediate second marginal.  Define
\begin{equation}\label{eq:def-nuRAk}
\nu^{R,\alpha,k}
:=
\nu^k\wedge_{\rm cd}\big(\mu^k\vee_{\rm cd}\nu^{R,\alpha}\big).
\end{equation}
Then Lemma~\ref{lem:put_minmax} gives
\begin{equation}\label{eq:nuRAk-conv}
W_1(\nu^{R,\alpha,k},\nu^{R,\alpha})\longrightarrow0,
\end{equation}
and, by construction,
\begin{equation}\label{eq:nuRAk-orders}
\mu^k\cvd\nu^{R,\alpha,k}\cvd\nu^k
\quad\text{for every }k.
\end{equation}
By ordinary stability of couplings in adapted Wasserstein distance, see
\cite[Proposition~2.3]{BJMP1}, there exists
\begin{equation}\label{eq:qk-def}
q^k=\mu^k(dx)q_x^k(dy)\in\Pi(\mu^k,\nu^{R,\alpha,k})
\end{equation}
such that
\begin{equation}\label{eq:qk-conv}
\mathcal{AW}_1(q^k,\pi^{R,\alpha})\longrightarrow0.
\end{equation}
Apply Lemma~\ref{lemm:3.4} (ii) to $q^k\to\pi^{R,\alpha}$ with
\[
A:=K,
\qquad
B:=\widetilde K:=\left(\frac{3a+\tilde a}{4},\,b+1\right),
\qquad
C:=L^\circ.
\]
We obtain 
\begin{equation}\label{eq:hatpi-def}
\mathcal M_1 \ni\hat\mu^k\le\mu^k, \qquad \mathcal M_1 \ni \hat\nu^k\le\nu^{R,\alpha,k},\qquad \hat\pi^k=\hat\mu^k(dx)\hat\pi_x^k(dy)
\in\Pi(\hat\mu^k,\hat\nu^k), 
\end{equation}
where $\hat\pi^k$ is 
concentrated on $\widetilde K\times L^\circ$.  Together with numbers $\theta_k\downarrow0$,
it holds that 
\begin{equation}\label{eq:hatpi-conv}
\mathcal{AW}_1\big(\hat\pi^k,(1-\theta_k)\pi^{R,\alpha}|_{K\times\R}\big)
\longrightarrow0.
\end{equation}
Let
\begin{equation}\label{eq:breve-nu-def}
\breve\nu^{R,\alpha}
:=
\text{second marginal of }\pi^{R,\alpha}|_{K\times\R}.
\end{equation}
We next reserve some mass of $\nu^{R,\alpha,k}$ lying strictly to the left of
$\widetilde K$.  By \eqref{eq:left-free-mass-limit}, choose a compact set
\[
L_-
\subset
\left(\ell,\frac{a+\tilde a}{2}\right)
\quad\text{such that}\quad
\nu^{R,\alpha}(L_-)>0.
\]
Choose $\tilde\ell,\tilde\rho\in I$ such that
\begin{equation}\label{eq:J-def}
\tilde\ell<\inf(L\cup L_-),
\qquad
\tilde\rho>\sup(L\cup\widetilde K),
\qquad
J:=[\tilde\ell,\tilde\rho]\Subset I.
\end{equation}
Set
\begin{equation}\label{eq:left-free-interval-def}
\widetilde L_-:=\left(\tilde\ell,\frac{a+\tilde a}{2}\right),
\qquad
\nu_-^k:=\nu^{R,\alpha,k}|_{\widetilde L_-}.
\end{equation}
Since $\widetilde L_-$ is open and $\nu^{R,\alpha}(\widetilde L_-)>0$,
\eqref{eq:nuRAk-conv} and Portmanteau's theorem imply that there are $\delta>0$ and
$k_0$ such that
\begin{equation}\label{eq:left-mass-lower}
\nu_-^k(\R)\ge\delta
\quad\text{for every }k\ge k_0.
\end{equation}
Let
\begin{equation}\label{eq:e-distance}
e:=\operatorname{dist}(\widetilde L_-,\widetilde K)>0.
\end{equation}

\subsection{Barycentre correction}\label{subsec:barycentre-correction}

The localised kernels $\hat\pi_x^k$ need not satisfy the supermartingale inequality.  We
correct them by mixing in a small amount of the left-free mass $\nu_-^k$.  For
$k\ge k_0$ and $\hat\mu^k$-a.e. $x$, set
\begin{equation}\label{eq:ck-dk-def}
A_k(x):=\left(\int y\,\hat\pi_x^k(dy)-x\right)^+,
\qquad
B_k(x):=\int (x-y)\,\nu_-^k(dy),
\end{equation}
\[
c_k(x):=\frac{A_k(x)}{B_k(x)},
\qquad
 d_k(x):=1+c_k(x)\nu_-^k(\R).
\]
By \eqref{eq:left-mass-lower} and \eqref{eq:e-distance},
\begin{equation}\label{eq:Bk-lower}
B_k(x)\ge e\nu_-^k(\R)\ge e\delta
\quad\text{for }x\in\widetilde K.
\end{equation}
Define
\begin{equation}\label{eq:def-tilde-kernel}
\tilde\pi_x^k
:=
\frac{\hat\pi_x^k+c_k(x)\nu_-^k}{d_k(x)},
\qquad
\tilde\pi^k:=\hat\mu^k(dx)\tilde\pi_x^k(dy),
\end{equation}
and denote the second marginal of $\tilde\pi^k$ by $\tilde\nu^k$.

\begin{proposition}\label{prop:correctSup}
For all sufficiently large $k$,
\begin{equation}\label{eq:tildepi-supermart}
\tilde\pi^k\in\Pi_S(\hat\mu^k,\tilde\nu^k).
\end{equation}
Moreover,
\begin{equation}\label{eq:tildepi-conv}
\mathcal{AW}_1\big(\tilde\pi^k,(1-\theta_k)\pi^{R,\alpha}|_{K\times\R}\big)
\longrightarrow0,
\end{equation}
and
\begin{equation}\label{eq:tilde-nu-dom}
(1-2\varepsilon)\tilde\nu^k
\le
(1-\varepsilon)\nu^{R,\alpha,k}
\quad\text{for all sufficiently large }k.
\end{equation}
\end{proposition}

\begin{proof}
The measure $\tilde\pi_x^k$ is a probability measure by definition.  If
$\int y\hat\pi_x^k(dy)\le x$, then $c_k(x)=0$.  Otherwise, the definition of $c_k(x)$
and the identity
$B_k(x)=x\nu_-^k(\R)-\int y\nu_-^k(dy)$ give
\[
\int y\,\tilde\pi_x^k(dy)=x.
\]
Thus \eqref{eq:tildepi-supermart} holds.

Since both $\nu_-^k/\nu_-^k(\R)$ and $\hat\pi_x^k$ are supported in
$J=[\tilde\ell,\tilde\rho]$, we can compare equal-mass measures and obtain
\begin{align*}
W_1(\tilde\pi_x^k,\hat\pi_x^k)
&\le
\frac{c_k(x)\nu_-^k(\R)}{d_k(x)}
W_1\left(\frac{\nu_-^k}{\nu_-^k(\R)},\hat\pi_x^k\right)\\
&\le
\frac{\tilde\rho-\tilde\ell}{e}\,A_k(x).
\end{align*}
By Lemma~\ref{lem:defect-vanishes} and \eqref{eq:hatpi-conv},
\[
\int A_k(x)\,\hat\mu^k(dx)\longrightarrow0.
\]
Hence $\mathcal{AW}_1(\tilde\pi^k,\hat\pi^k)\to0$, and \eqref{eq:tildepi-conv} follows
from \eqref{eq:hatpi-conv}.

For the domination of the second marginal, write
\[
\tilde\nu^k
=
\int \frac1{d_k(x)}\hat\pi_x^k\,\hat\mu^k(dx)
+
\nu_-^k\int \frac{c_k(x)}{d_k(x)}\,\hat\mu^k(dx).
\]
The first term is bounded above by $\hat\nu^k\le\nu^{R,\alpha,k}$, and
$\nu_-^k\le\nu^{R,\alpha,k}$.  Furthermore,
\[
r_k:=\int \frac{c_k(x)}{d_k(x)}\,\hat\mu^k(dx)
\le
\frac1{e\delta}\int A_k(x)\,\hat\mu^k(dx)
\longrightarrow0.
\]
Thus $\tilde\nu^k\le(1+r_k)\nu^{R,\alpha,k}$.  For all large $k$,
$(1-2\varepsilon)(1+r_k)\le1-\varepsilon$, which gives \eqref{eq:tilde-nu-dom}.
\end{proof}

\subsection{Completion of the residual marginals}\label{subsec:completion-gluing}

We now complete the sub-coupling $(1-2\varepsilon)\tilde\pi^k$ to a supermartingale
coupling whose first marginal is $\mu^k$ and whose second marginal is
$\varepsilon\nu^k+(1-\varepsilon)\nu^{R,\alpha,k}$.  Define
\begin{equation}\label{eq:residual-marginals-def}
\mu_{\rm rem}^k:=\mu^k-(1-2\varepsilon)\hat\mu^k,
\qquad
\nu_{\rm rem}^k:=
\varepsilon\nu^k+(1-\varepsilon)\nu^{R,\alpha,k}-(1-2\varepsilon)\tilde\nu^k.
\end{equation}
By \eqref{eq:tilde-nu-dom}, $\nu_{\rm rem}^k$ is a positive measure for all large $k$.
Let us prove
\begin{equation}\label{eq:rem-cvd}
\mu_{\rm rem}^k\cvd\nu_{\rm rem}^k.
\end{equation}
Equivalently,
\begin{equation}\label{eq:residual-put-goal}
P_{\varepsilon\nu^k+(1-\varepsilon)\nu^{R,\alpha,k}}-P_{\mu^k}
\ge
P_{(1-2\varepsilon)\tilde\nu^k}-P_{(1-2\varepsilon)\hat\mu^k}
\quad\text{on }\R.
\end{equation}
We prove \eqref{eq:residual-put-goal} on $J$ and on $J^c$ separately.

\begin{proposition}\label{prop:reduceTOjc}
For all sufficiently large $k$, inequality \eqref{eq:residual-put-goal} holds on $J$.
\end{proposition}

\begin{proof}
Let $\breve\nu^{R,\alpha}$ be defined by \eqref{eq:breve-nu-def}.  Consider the limiting
gap
\begin{align*}
G(x)
&:=P_{\varepsilon\nu+(1-\varepsilon)\nu^{R,\alpha}}(x)-P_\mu(x)
   -P_{(1-2\varepsilon)\breve\nu^{R,\alpha}}(x)
   +P_{(1-2\varepsilon)\mu|_K}(x)\\
&=\varepsilon(P_\nu-P_\mu)(x)
  +\varepsilon(P_{\breve\nu^{R,\alpha}}-P_{\mu|_K})(x) +(1-\varepsilon)
   \big(P_{\nu^{R,\alpha}-\breve\nu^{R,\alpha}}-P_{\mu|_{K^c}}\big)(x).
\end{align*}
The last two terms are nonnegative because the restrictions of $\pi^{R,\alpha}$ to
$K\times\R$ and $K^c\times\R$ are supermartingale couplings.  Since
$J\Subset I$ and $P_\nu-P_\mu>0$ on $I$,
\[
\tau:=\inf_{x\in J}G(x)>0.
\]
By \eqref{eq:nuRAk-conv}, \eqref{eq:tildepi-conv}, and the convergence of the first
marginals, the corresponding finite-measure put functions converge uniformly on $J$.
Therefore, for all large $k$, the gap in \eqref{eq:residual-put-goal} is at least
$\tau/2$ on $J$.
\end{proof}

\begin{proposition}\label{prop:potentCompare}
For all sufficiently large $k$,
\begin{equation}\label{eq:tail-comparison-strong}
P_{\nu^{R,\alpha,k}}-P_{\mu^k}
\ge
P_{(1-2\varepsilon)\tilde\nu^k}-P_{(1-2\varepsilon)\hat\mu^k}
\quad\text{on }J^c.
\end{equation}
Consequently, \eqref{eq:residual-put-goal} holds on $J^c$.
\end{proposition}

\begin{proof}
Set
\[
D^{R,\alpha}(x):=P_{\nu^{R,\alpha}}(x)-P_\mu(x),
\qquad
D^k(x):=P_{\nu^{R,\alpha,k}}(x)-P_{\mu^k}(x),
\]
and
\[
\Delta^k:=\Delta(\mu^k,\nu^{R,\alpha,k}).
\]
By \eqref{eq:nuRAk-orders}, $D^k\ge0$.  Moreover, by \eqref{eq:RA-defect},
$\Delta^k\to\Delta^{R,\alpha}>0$, and
\[
D^{R,\alpha}(x)\longrightarrow\Delta^{R,\alpha}
\quad\text{as }x\to+\infty.
\]
Increasing $\tilde\rho$ in \eqref{eq:J-def} if necessary, and using the uniform convergence
$D^k\to D^{R,\alpha}$, we may assume that, for all large $k$,
\begin{equation}\label{eq:Dk-tail-lower}
D^k(x)
\ge
\left(1-\frac{\varepsilon}{2}\right)\Delta^k,
\qquad x\ge\tilde\rho.
\end{equation}
On the other hand, \eqref{eq:tildepi-conv} gives
\[
\hat\Delta^k:=\Delta(\hat\mu^k,\tilde\nu^k)
\longrightarrow
\Delta_K^{R,\alpha}:=
\Delta(\mu|_K,\breve\nu^{R,\alpha})\le \Delta^{R,\alpha}.
\]
Since $\Delta^{R,\alpha}>0$ and $\varepsilon>0$, it follows that, for all large $k$,
\begin{equation}\label{eq:global-local-defect}
\left(1-\frac{\varepsilon}{2}\right)\Delta^k
\ge
(1-2\varepsilon)\hat\Delta^k.
\end{equation}
If $x\le\tilde\ell$, then $\hat\mu^k$ and $\tilde\nu^k$ are supported in $J$, so the right
side of \eqref{eq:tail-comparison-strong} is zero, while $D^k(x)\ge0$.  If
$x\ge\tilde\rho$, then the same support property gives
\[
P_{(1-2\varepsilon)\tilde\nu^k}(x)-P_{(1-2\varepsilon)\hat\mu^k}(x)
=(1-2\varepsilon)\hat\Delta^k.
\]
Combining this identity with \eqref{eq:Dk-tail-lower} and
\eqref{eq:global-local-defect} proves \eqref{eq:tail-comparison-strong}.
\end{proof}

By Propositions~\ref{prop:reduceTOjc} and \ref{prop:potentCompare},
\eqref{eq:rem-cvd} holds.  Strassen's theorem gives
\[
\eta^k\in\Pi_S(\mu_{\rm rem}^k,\nu_{\rm rem}^k).
\]
Define
\begin{equation}\label{eq:barpi-def}
\bar\pi^{k,\varepsilon}:=(1-2\varepsilon)\tilde\pi^k+\eta^k.
\end{equation}
Then
\begin{equation}\label{eq:barpi-marginals}
\bar\pi^{k,\varepsilon}
\in
\Pi_S\big(\mu^k,\lambda^{k,\varepsilon}\big),
\qquad
\lambda^{k,\varepsilon}:=
\varepsilon\nu^k+(1-\varepsilon)\nu^{R,\alpha,k}.
\end{equation}
Furthermore, by \eqref{eq:K-choice}, \eqref{eq:RA-close}, \eqref{eq:tildepi-conv}, and
\cite[Lemma~3.6(b)]{BJMP1}, there is a deterministic modulus $\omega(\varepsilon)$,
with $\omega(\varepsilon)\downarrow0$ as $\varepsilon\downarrow0$, such that
\begin{equation}\label{eq:barpi-close}
\limsup_{k\to\infty}
\mathcal{AW}_1(\bar\pi^{k,\varepsilon},\pi)
\le
\omega(\varepsilon).
\end{equation}
Also, by \eqref{eq:nuRAk-conv} and the choice of $R,\alpha$,
\begin{equation}\label{eq:lambda-close}
\limsup_{k\to\infty}W_1(\lambda^{k,\varepsilon},\nu)
\le
\omega(\varepsilon).
\end{equation}
Finally, \eqref{eq:nuRAk-orders} implies
\begin{equation}\label{eq:lambda-order}
\lambda^{k,\varepsilon}\cvd\nu^k
\quad\text{for all }k.
\end{equation}

\subsection{Final adjustment of the second marginal}\label{subsec:final-adjust}

The last step changes the second marginal from $\lambda^{k,\varepsilon}$ to $\nu^k$.  The
important point is that this must be done by a supermartingale coupling, not by a
martingale coupling, since in general
$\lambda^{k,\varepsilon}$ is only below $\nu^k$ in decreasing-convex order.

Choose a sequence $\varepsilon_j\downarrow0$.  By \eqref{eq:barpi-close} and
\eqref{eq:lambda-close}, we can choose integers
$1\le N_1<N_2<\cdots$ such that, for every $j$ and every $k\ge N_j$,
\begin{equation}\label{eq:diagonal-choice}
\mathcal{AW}_1(\bar\pi^{k,\varepsilon_j},\pi)
\le
\omega(\varepsilon_j)+\varepsilon_j,
\qquad
W_1(\lambda^{k,\varepsilon_j},\nu)
\le
\omega(\varepsilon_j)+\varepsilon_j.
\end{equation}
For $N_j\le k<N_{j+1}$, set
\[
\dot\pi^k:=\bar\pi^{k,\varepsilon_j},
\qquad
\lambda^k:=\lambda^{k,\varepsilon_j}.
\]
For the finitely many $k<N_1$, choose arbitrary admissible couplings.  Then
\begin{equation}\label{eq:dotpi-lambda-conv}
\dot\pi^k\in\Pi_S(\mu^k,\lambda^k),
\qquad
\mathcal{AW}_1(\dot\pi^k,
\pi)\longrightarrow0,
\qquad
W_1(\lambda^k,
\nu)\longrightarrow0.
\end{equation}
Moreover, $\lambda^k\cvd\nu^k$ by \eqref{eq:lambda-order}.  Hence Strassen's theorem
provides a supermartingale coupling
\[
M^k\in\Pi_S(\lambda^k,\nu^k).
\]
Since both marginals of $M^k$ converge to $\nu$ in $W_1$, Lemma~\ref{lem:diagonal-component}
applied to $M^k$ yields
\begin{equation}\label{eq:final-kernel-diagonal}
\int_{\R^2}|z-y|\,M^k(dz,dy)
\longrightarrow0.
\end{equation}
Disintegrate $M^k(dz,dy)=\lambda^k(dz)M_z^k(dy)$ and define
\begin{equation}\label{eq:final-pik-def}
\pi^k(dx,dy):=
\int_{\R}\dot\pi^k(dx,dz)M_z^k(dy).
\end{equation}
Then $\pi^k$ has marginals $\mu^k$ and $\nu^k$.  It is a supermartingale coupling because
\[
\int y\,M_z^k(dy)\le z
\quad\text{and}\quad
\int z\,\dot\pi_x^k(dz)\le x.
\]
Finally,
\[
\mathcal{AW}_1(\pi^k,\dot\pi^k)
\le
\int_{\R^2}|z-y|\,M^k(dz,dy)
\longrightarrow0.
\]
Together with \eqref{eq:dotpi-lambda-conv}, this gives
\[
\mathcal{AW}_1(\pi^k,\pi)\longrightarrow0.
\]
This proves Theorem~\ref{thm:approx} in the strict irreducible case, and therefore, by
the reduction above, in full generality.

\section{Proof of Theorem \ref{thm:monotonicity}: monotonicity principle} \label{sec:monoto_proof}

We now proceed to prove the monotonicity principle for (WSOT). We will first use the stability result to prove the sufficiency side, and then use an adapted version of an abstract result in \cite{BeiglbockPlenty} to prove the necessity side.

\subsection{Sufficiency}

In order to prove the sufficiency, we first argue the following sufficiency for finite optimality.

\begin{lemma} \label{lemma:fini_opt}
  If $\pi \in \Pi_S(\mu, \nu)$ is a finitely supported coupling of the form
    $\frac{1}{N} \sum_{i=1}^N \delta_{(x_i)}(dx) p_i (dy)$  for  $x_1 < \cdots < x_n \in \R$   and  $p_1, \cdots, p_n \in \mathcal{P}_1(\R)$, and is supermartingale $C$-monotone w.r.t. some $\Gamma$, then it is optimal among all $\pi'\in \Pi_S(\mu,\nu)$ with the same form.
\end{lemma}

\begin{proof}
Let  any martingale coupling $\pi' \in \Pi_S(\mu, \nu)$ of the form
 $\frac{1}{N} \sum_{i=1}^N \delta_{(x_i)}(dx) q_i (dy)$  for  $x_1 < \cdots < x_n \in \R$   and  $q_1, \cdots, q_n \in \mathcal{P}_1(\R)$,
    and such that $\sum_{i=1}^N p_i = \sum_{i=1}^N q_i$. Now define $M_0, M_1$ as the martingale region of $\pi$ to the left and right of $x^*$.
    
It is clear that $\forall i \in \{1, \ldots, N\}$, $\int_{\R} y p_i (dy) = x_i = \int_{\R} y q_i (dy)$ for all $x_i \in M_1$ and $\forall i \in \{1, \ldots, N\}$, $\int_{\R} y p_i (dy) = x_i \geq \int_{\R} y q_i (dy)$ for all $x_i \in M_0$. By definition of supermartingale $C$-monotonicity, we get
    \[
    \int_{\R \times \R} C(x, \pi_x)  \mu(dx ) = \frac{1}{N} \sum_{i=1}^N C(x_i, p_i) \leq \frac{1}{N} \sum_{i=1}^N C(x_i, q_i) = \int_{\R \times \R} C(x, \pi'_x)  \mu(dx ).
    \]
\end{proof}

\begin{proposition} \label{prop:sufficiency} (Sufficiency)
 Let $r \geq 1$ and $\mu, \nu \in \Pc_r(\R)$ be in convex decreasing order, and $C:\R\times\sP_r(\R) \to \R$ be a measurable cost function, continuous in the second argument and such that there exists a finite constant $K$ which satisfies, for all $(x,m)\in\R\times\sP_r(\R)$,
$$
\lvert C(x,m)\lvert\leq K\left(1+\lvert x\lvert^r+\int_\R\lvert y\lvert^rm(dy)\right).
$$
%Fix $\mu,\nu\in\sP_r(\R)$
%  Let $\mu \in \Pc_f(\R)$, $\nu \in \Pc_g(\R)$ be in convex decreasing order, and $C:\R \times\Pc(\R) \to \R$ be a measurable cost %function, continuous in the second argument and such that there exists a finite constant $K$ which satisfies
%$$
%\forall (x,p) \in \R \times \Pc_g(\R), \ C(x,p) \leq K \left( f(x)  + \int_{\R} g(y)  p (dy) \right).
%$$
Let  $\pi \in \Pi_S(\mu, \nu)$ be supermartingale $C$-monotone w.r.t. some $\Gamma$. Then $\pi$ is optimal for (WSOT).

\end{proposition}

\begin{proof}
The argument is similar as \cite[Theorem 3.3]{BJMP}, here we state it in a slightly simpler context.

Step 1. Suppose $\mu$ is concentrated on a Polish subspace $X$ of $\R$, and the restriction $C|_{X \times \sP_r(\R)}$ is continuous. Then combine Lemma \ref{lemma:fini_opt} with Theorem \ref{thm:stability} we get $\pi$ is optimal for $V^C_S(\mu,\nu)$.

Step 2. We lift the requirements of $C$ being continuous. Using similar approximation arguments with Lusin's theorem as \cite[Theorem 3.3]{BJMP}, and replace \cite[Lemma A.6]{BJMP} with \cite[Lemma 5.7]{Carda}, we can draw the conclusion.
\end{proof}

\subsection{Necessity}

In the following, we justify the necessity side of the monotonicity principle. First, we need an abstract version of this theorem similar as \cite{BeiglbockPlenty} with equalities in the linear constraints being replaced by inequalities. Then, we need to apply the abstract theorem in a suitable context to get the monotonicity principle for (WSOT). For completeness we state the main results here in this Appendix. Throughout this section, we denote $X, Y$ to be two Polish spaces.

Let $E$ be a Polish space and $c: E\rightarrow \R$ be Borel measurable. Given a set $\Fc$ of Borel functions on $E$, such that $\Fc= \Fc_1 \cup \Fc_2$, denote $\Pi_{\Fc}$ 
 the set of probability measures $\gamma$ on $E$ for which $\int f d\gamma = 0$ for all $f \in \Fc_1$, and $\int f d\gamma \leq 0$ for all $f \in \Fc_2$. In the following, we consider the constrained optimization problem is defined by:
 $$
 \min_{\gamma \in \Pi_{\Fc}} \int c d \gamma.
 $$
For $\alpha \in \Mc(E)$, $\alpha'\in E$ is a \textit{competitor} of $\alpha$ if $\alpha(E)=\alpha'(E)$, and for all $f \in \Fc_1$, $\int f d \alpha = \int f d \alpha'$; and for all $f \in \Fc_2$, $\int f d \alpha \geq \int f d \alpha'$.
In addition, if $\alpha$ is finitely supported, its competitor should also be so.
A set $\Gamma \subset E$ is called $c$-monotone (or finitely minimal) if each measure $\alpha$, which is finite and concentrated on finitely many points in $\Gamma$, is cost minimizing among its competitors. A measure $\gamma$ is called  $c$-monotone if it is concentrated on a $c$-monotone set.

We first state the adaptation of \cite[Theorem 1.4]{BeiglbockPlenty} into the current text, with the constraint being inequality.

\begin{theorem} \label{thm:monoAbstract}
Suppose there exists a function $g: E \to [0,\infty)$ such that for all $f \in \Fc$, there exists a constant $a_f \in \R_+$ s.t. $|f| \leq a_f g$. In addition, all functions in $\Fc$ are continuous, or $\Fc$ is at most countable. Assume further $\gamma^*$ is the optimizer of the constrained problem:
$$
\min_{\gamma \in \Pi_{\Fc}} \int c d \gamma =  \int c d \gamma^* \in \R.
$$
Then $\gamma^*$ is $c$-monotone.
\end{theorem}
\begin{proof}
As in \cite[Theorem 1.4]{BeiglbockPlenty}, the main idea is to apply a Kellerer's measure-theoretic duality result \cite[Lemma 4.1]{BeiglbockPlenty}.  Define the set $M$ similar as \cite[Theorem 1.4]{BeiglbockPlenty} with a supermartingale version of $c$-better competitor. Also define $\hat M$ similar but with
$$
\sum \alpha'_i f (z'_i) \leq \sum \alpha_i f (z_i),
$$
instead of the equality. Applying the duality  to $M$: if item (i) in the duality holds, then the proof is done. Otherwise, if item (ii) holds, one try to draw a contradiction.

Indeed, by Jankov-von Neumann measurable selection theorem (see \cite[Proposition 7.49]{BertShreve} ) to the set $\hat{M}$, one can find a mapping
$$
z \mapsto (\alpha_1(z), \cdots, \alpha_l(z), z'_1(z), \cdots, z'_l(z), \alpha'_1(z), \cdots, \alpha'_l(z))
$$
such that
$$
(z, \alpha_1(z), \cdots, \alpha_l(z), z'_1(z), \cdots, z'_l(z), \alpha'_1(z), \cdots, \alpha'_l(z)) \in \hat{M}.
$$
Then argue similarly as \cite[Theorem 1.4]{BeiglbockPlenty}, one can then define the measures $\om$ and $\om'$, and $\om'$ is a $c$-better competitor of $\om$:
$$
\int f d \om' = \int f d \om, \forall f \in \Fc_1, \ \int f d \om' \leq \int f d \om, \forall f \in \Fc_2,  \ \ \ \ \int c d \om' < \int c d \om.
$$
Consequently we can construct another probability measure $\gamma':=\gamma^* - \om + \om' \in \Pi_{\Fc}$,  which is also a contradiction to the optimality of $\gamma^*$.
%{\color{blue} understand the upper bound $l$ in all place? }
%set $\hat{M}$ same; projection $M$ different? no, ``better competitor" contains the inequality
\end{proof}

Remind the definition of $\Lambda_S(\mu,\nu)$ in Section \ref{subsec:alternativeWSOT}.
Expressed in terms of linear constraints, we have $\P \in \Lambda_S(\mu,\nu)$ iff
$$
\int_{\R \times \Pc(\R)} f(x,p) \P(dx, dp) \leq 0, \ \ \forall f \in \Fc^s_2,
$$
where
\[
\mathcal F_2^s
:=
\left\{
f_{\varphi}(x,p)
=
\varphi(x,p)
\left(
\int_{\mathbb R} y\,p(dy)-x
\right)
:
\varphi\in C_b^+\bigl(\mathbb R\times \mathcal P_1(\mathbb R)\bigr)
\right\}.
\]
and
$$
\int_{\R \times \Pc(\R)} f(x,p) \P(dx, dp) = 0, \ \ \forall f \in \Fc^s_1,
$$
where
$$
\Fc^s_1: = \{ \varphi \circ p_X - \int \varphi d \mu: \varphi \in C_b(X) \} \cup \{ \varphi \circ p_Y - \int \varphi d \nu: \varphi \in C_b(X) \}.
$$
Now we are able to provide the necessity part of the monotonicity principle.

\begin{proposition} \label{prop:necessity} (Necessity) 

 Let $r\geq 1$, and $\mu\cvd \nu \in \Pc_r(\R)$ be in convex decreasing order, and $C:\R\times\sP_r(\R) \to \R$  be a measurable cost function, continuous and convex in the second argument and such that there exists a finite constant $K$ which satisfies, for all $(x,m)\in\R\times\sP_r(\R)$,
$$
\lvert C(x,m)\lvert\leq K\left(1+\lvert x\lvert^r+\int_\R\lvert y\lvert^rm(dy)\right).
$$
Let $\pi^* \in \Pi_S(\mu, \nu)$ be a supermartingale coupling which minimises (WSOT), then $\pi^*$ is a Supermartingale $C$-monotone coupling.

\end{proposition}

\begin{proof}

The argument is similar as \cite[Theorem 3.4]{BackPam}, and we include it here for completeness.
First we notice that (WSOT') can be seen as an optimal transport problem with addition linear constraints taking inequality form. By Theorem \ref{thm:monoAbstract}, we have if $P^* \in \Lambda_S(\mu,\nu)$, then $P^*$ is supermartingale $C$-monotone. Recall that the supermartingale $C$-monotone is defined in Definition \ref{def:supMart_mono}.

Notice that any supermartingale coupling $\pi \in \Pi_s(\mu,\nu)$ induces an element in $\Lambda_S(\mu,\nu)$ by the embedding $J$ defined before this proposition; so the function value of (WSOT') is smaller than (WSOT). When cost function is further convex as in the current context, by Proposition \ref{Prop:A9}, we have  $C(x, I(\Q)) \leq \int_{\Pc(\R)} C(x,p) \Q(dp)$. Let $\P \in \Lambda_S(\mu,\nu)$, then $\mu(dx) I(\P_x) \in \Pi_S(\mu,\nu)$. By above inequality,
$$
\int_{\R \times \Pc(\R)} C(x,p) \P(dx, dp) \geq \int_{\R \times \Pc(\R)} C(x, I(\P_x)) \mu(dx).
$$
Consequently the function values of (WSOT') and (WSOT) are equal. As $\pi^*$ is optimal for (WSOT), $J(\pi^*)$ is optimal for (WSOT'); and we deduce by first paragraph that $J(\pi^*)$ is supermartingale C-monotone. Similar reasoning as \cite[Remark 2.4(a)]{BackPam} leads to $\pi$ martingale C-monotone iff $J(\pi)$ martingale C-monotone. 
Consequently, we have $\pi^*$ is also supermartingale C-monotone.
\end{proof}

\appendix

\section{Appendix} \label{sec:appendix}

In this Appendix, we collect some needed results.

\subsection{Irreducible decomposition for supermartingale couplings}\label{sec:irreducible}

Let $\mu,\nu\in\mathcal P_1$ with $\mu\cvd\nu$ and set $D:=P_\nu-P_\mu$.
Define
\[
x^*:=\sup\{x\in\mathbb R:\ D(x)=0\}.
\]
Roughly speaking, points where $D$ vanishes act as barriers: for $\pi\in\Pi_S(\mu,\nu)$,
no mass can cross such a point on the ``martingale side''; see \cite{NutzStebegg.18}.
This yields an irreducible decomposition into components where the problem is genuinely supermartingale on the right
and martingale on the left.

\begin{lemma}[Nutz--Stebegg \cite{NutzStebegg.18},Proposition 3.4]\label{lem:irreducible}
Let $\mu,\nu\in\mathcal P_1$ with $\mu\cvd\nu$.
Let $I_0:=(x^*,+\infty)$, let $(I_k)_{k\geq1}$ be the open connected components of $\{D>0\}\cap(-\infty,x^*)$,
and set $I_{-1}:=\mathbb R\setminus\bigcup_{k\geq0}I_k$.
Let $\mu_k:=\mu|_{I_k}$ for $k\geq-1$, so that $\mu=\sum_{k\geq-1}\mu_k$.

Then there exists a unique decomposition $\nu=\sum_{k\geq-1}\nu_k$ such that
\[
\mu_{-1}=\nu_{-1},\quad \mu_0\cvd\nu_0\quad\text{and}\quad \mu_k\cv\nu_k\ \text{for all }k\geq1.
\]
Furthermore, any $\pi\in\Pi_S(\mu,\nu)$ admits a unique decomposition $\pi=\sum_{k\geq-1}\pi_k$ such that
$\pi_0\in\Pi_S(\mu_0,\nu_0)$ and $\pi_k\in\Pi_M(\mu_k,\nu_k)$ for all $k\neq0$.
\end{lemma}

\subsection{A localisation lemma in adapted Wasserstein topology}\label{subsec:localisation}

The following result (from \cite{BJMP1}) allows one to localise adapted Wasserstein convergence on sets of positive mass.
It is a key technical input for our approximation argument.

\begin{lemma}[\cite{BJMP1}, Lemma 3.4]\label{lemm:3.4}
Let $r\ge1$. Let $\mu,\mu^k\in\mathcal M_r(X)$ and $\nu,\nu^k\in\mathcal M_r(Y)$ have equal masses, and assume that
\[
\pi^k\in\Pi(\mu^k,\nu^k)\ \xrightarrow{\mathcal{AW}_1}\ \pi\in\Pi(\mu,\nu).
\]
Let $A\subset X$ be measurable and let $B\supset A$ be open.

\vspace{1mm}

\noindent \rm(i)\,
Let $\gamma^k\in\Pi(\mu^k,\mu)$ be an optimizer in \eqref{eq:AW_def} for $\AW(\pi^k,\pi)$ and set
\[
\tilde\mu^k:=\gamma^k(\cdot\times A),
\qquad
\varepsilon_k:=1-\frac{\tilde\mu^k(X)}{\mu(A)}.
\]
Then $\tilde\mu^k\le \mu^k|_B$ and $\varepsilon_k\ge0$. Moreover, defining
\[
\tilde\pi^k:=\tilde\mu^k\times \pi_x^k,
\]
one has
\[
\mathcal{AW}_1\!\big(\tilde\pi^k,\ (1-\varepsilon_k)\,\pi|_{A\times Y}\big)+\varepsilon_k\longrightarrow0.
\]
\noindent \rm(ii)\,
Suppose that $\nu$ is concentrated on some $C\subset Y$.
Let $\tilde\nu^k$ and $\tilde\nu$ be the second marginals of $\tilde\pi^k$ and $\pi|_{A\times Y}$, respectively, and define
\[
\hat\mu^k=\frac{\tilde\nu^k(C)}{\tilde\mu^k(X)}\,\tilde\mu^k,
\qquad
\varepsilon_k':=1-\frac{\tilde\nu^k(C)}{\mu(A)}.
\]
Then there exist $\hat\nu^k\le\nu^k$ and $\hat\pi^k=\hat\mu^k\times\hat\pi_x^k\in\Pi(\hat\mu^k,\hat\nu^k)$
concentrated on $B\times C$ such that
\[
\mathcal{AW}_1^r \!\big(\hat\pi^k,\ (1-\varepsilon_k')\,\pi|_{A\times Y}\big)
+\int_X W_r^r(\hat\pi_x^k,\pi_x^k)\,\hat\mu^k(dx)
+\varepsilon_k'\longrightarrow0.
\]
In particular, $1-\varepsilon_k'\le 1-\varepsilon_k$.
\end{lemma}

\subsection{An alternative form of WSOT} \label{subsec:alternativeWSOT}

To apply Theorem~\ref{thm:monoAbstract}, it is convenient to work with the canonical embedding $J$: 
$$
\begin{aligned}
J:& \Pc(X \times Y) \to \Pc(X \times \Pc(Y)),\\
&\pi \mapsto \mbox{proj}_X(\pi)(dx) \delta_{\pi_x}(dp).
\end{aligned}
$$
Consider its left-inverse, the $X\times Y$-intensity map $\hat I$, and the intensity map $I$:
$$
\begin{aligned}
\hat I : \ & \Pc(X \times  \Pc(Y)) \to \Pc(X \times Y), \ \ \ \  I :  & \Pc( \Pc(Y)) \to  \Pc( Y),\hspace{30mm} \\
&\P \mapsto  \int_{p \in \Pc(Y)} p(dy) \P(dx, dp), &\Q \mapsto  I(\Q)(dy):= \int_{ \Pc(Y)} p(dy) \Q(dp).
\end{aligned}
$$
We further define the set $\Lambda(\mu,\nu) := \left\{ \P \in \Pc(X \times \Pc(Y)): \hat I(\P) \in \Pi(\mu, \nu) \right\}$.

\begin{proposition} \label{Prop:A9} (\cite[Proposition A.9]{BJMP}) 
Let $g: Y \to [1, +\infty)$ be continuous, $C: \Pc_g(Y) \to \R$ be convex, lower semicontinuous and lower bounded by a negative multiple of $\hat g$. Then for all $\Q \in \Pc_{\hat g}(\Pc(Y))$ holds
$$
C(I(\Q)) \leq \int_{\Pc_g(Y)} C(p) \Q(dp).
$$
\end{proposition}

Now we can introduce the alternative formulation (WSOT'):
$$
\inf_{\pi \in \Lambda_S(\mu,\nu)} \int_{\R \times \Pc(\R)} C(x, p) P(dx, dp),
$$
where $\Lambda_S(\mu,\nu)$ is the set of all $\P \in \Lambda(\mu,\nu)$ with full measure on $\{ (x,p) \in \R \times \Pc_1: x \geq \int_{\R} y p(dy)  \}$.

\subsection{Competitor and finite optimality}\label{subsec:def_mono}

To formulate the monotonicity principle (Section~\ref{subsec:mono}), we recall the notions of competitor and finite optimality.

\begin{definition}[Competitor]\label{def:competitor}
Let $\pi$ be a finite measure on $\R^2$ with finite first moment, and let $\pi_1$ be its first marginal.
Fix a disintegration $\pi=\pi_1\times\kappa_x$.
Let $M_0,M_1\subset\R$ be Borel.
A measure $\pi'$ is an \emph{$(M_0,M_1)$-competitor} of $\pi$ if it has the same marginals as $\pi$ and if,
writing $\pi'=\pi_1\times\kappa'_x$, one has
\[
\int_\R y\,\kappa'_x(dy)\le \int_\R y\,\kappa_x(dy)\quad\text{for }\pi_1\text{-a.e.\ }x\in M_0,
\qquad
\int_\R y\,\kappa'_x(dy)= \int_\R y\,\kappa_x(dy)\quad\text{for }\pi_1\text{-a.e.\ }x\in M_1.
\]
\end{definition}

\begin{definition}[Finite optimality]\label{def:finite_opt}
Let $G:\R\times\R\to\R$ be a Borel cost function and let $\Gamma\subset\R^2$ be Borel.
We say that $\Gamma$ is \emph{finitely optimal for $G$} if for every finitely supported probability measure
$\alpha\in\mathcal P(\R^2)$ with $\supp(\alpha)\subset\Gamma$, one has
\[
\int_{\R^2} G(x,y)\,\alpha(dx,dy)
\le
\int_{\R^2} G(x,y)\,\alpha'(dx,dy)
\]
for every competitor $\alpha'$ of $\alpha$ (in the sense relevant to the constraint under consideration).
\end{definition}

\bibliographystyle{plain}
\bibliography{references}

\end{document}